\documentclass[reqno]{amsart}

\usepackage{amsmath,amssymb,latexsym,soul,cite,mathrsfs}
\usepackage{enumitem,graphicx}
\usepackage[dvipsnames]{xcolor}
\usepackage{soul}
\usepackage{xcolor, soul}
\usepackage{xcolor}
\usepackage{hyphenat}

\usepackage[colorlinks=true,urlcolor=blue,citecolor=red,linkcolor=blue,linktocpage,pdfpagelabels,bookmarksnumbered,bookmarksopen]{hyperref}
\usepackage[hyperpageref]{backref}

\newcommand{\ep}{\varepsilon}

\newcommand{\der}{\mathbb{D}}
\newcommand{\inte}{\mathrm{I}}
\newcommand{\mitt}{\mathscr{E}}
\newcommand{\nat}{\mathbb{N}}

\usepackage[left=2.5cm,right=2.5cm,top=2cm,bottom=2cm]{geometry}

\numberwithin{equation}{section}

\newtheorem{theorem}{Theorem}[section]
\newtheorem{lemma}[theorem]{Lemma}
\newtheorem{corollary}[theorem]{Corollary}

\newtheorem{remark}[theorem]{Remark}
\newtheorem{definition}[theorem]{Definition}
\newtheorem{example}[theorem]{Example}

\title{ On the $\varepsilon$-regular mild solution for fractional abstract integro-differential equations}

\author[J. Vanterler C. Sousa,  M. Aurora P. Pulido, V. Govindaraj and E. Capelas de Oliveira]{J. Vanterler C. Sousa,  M. Aurora P. Pulido, V. Govindaraj and E. Capelas de Oliveira}

\address[J. Vanterler da C. Sousa]{Universidade Federal do ABC, Centro de Matemática, Computação e Cognição, 09.210-580, Santo André/SP,  Brazil
}
\email{\href{mailto:jose.vanterler@ufabc.edu.br}{jose.vanterler@ufabc.edu.br}}

\address[M. Aurora P. Pulido]{Department of Applied Mathematics, Imecc-State University of Campinas, Campinas SP13083-859, Brazil}
\email{\href{mailto:ra211681@ime.unicamp.br}{ra211681@ime.unicamp.br}}

\address[V. Govindaraj ]{Department of Mathematics, National Institute of Technology Puducherry, Karaikal 609 609, India}
\email{\href{mailto:govindaraj.maths@gmail.com}{govindaraj.maths@gmail.com}}

\address[E. Capelas de Oliveira]{Department of Applied Mathematics, Imecc-State University of Campinas, Campinas SP13083-859, Brazil}
\email{\href{capelas@unicamp.br}{capelas@unicamp.br}}


\begin{document}

\begin{abstract}

In this present paper, we first obtained some estimates involving parts of $\varepsilon$-regular mild solutions of the fractional integro-differential equation. In this sense, through these preliminary results, we investigate the main results of this paper, i.e., the existence, regularity and continuous dependence of $\varepsilon$-regular mild solutions for fractional abstract integro-differential equations in Banach space.
    
\end{abstract}
\maketitle
\section{Introduction and motivation}

This paper is concerned with the fractional integro-differential equation given by
\begin{equation}\label{Eq1}
\left\{ 
\begin{array}{cll}
^{H}\der_{0+}^{\alpha,\beta}\zeta(u) & = & \mathcal{A} \zeta(u)+\displaystyle\int_{0}^{u}g(u-s,\zeta(s))ds+f(u,\zeta(u)),\quad u>0 \\ 
\inte_{0+}^{1-\gamma}\zeta(0) & = & \zeta_0\in \mathcal{D}(\mathcal{A}),
\end{array}
\right.
\end{equation}
where $\mathcal{A}:\mathcal{D}(\mathcal{A})\subset \mathbb{B}_{0}\rightarrow \mathbb{B}_{0}$ is a linear operator such that, $\mathcal{A}$ is a sectorial operator, $\mathbb{B}_{0}$ is a Banach space and $g$ and $f$ are functions satisfying certain conditions. $^{H}\der_{0+}^{\alpha,\beta}(\cdot)$ is a Hilfer fractional derivative (HFD) of order $\alpha$ $(0<\alpha<1)$ and type $\beta$ $(0\leq\beta\leq 1)$ and $\inte_{0+}^{1-\gamma}(\cdot)$ is a Riemann-Liouville fractional integral (RLFI) of order $1-\gamma$ with $\gamma=\alpha+\beta(1-\alpha)$.

For $g(u-s,\zeta(s))=0$ in Eq. (\ref{Eq1}), we have the fractional differential equation (FDE) given by
\begin{equation*}
\left\{ 
\begin{array}{cll}
^{H}\der_{0+}^{\alpha,\beta}\zeta(u) & = & \mathcal{A} \zeta(u)+f(u,\zeta(u)),\quad u>0 \\ 
\inte_{0+}^{1-\gamma}\zeta(0) & = & \zeta_0\in \mathcal{D}(\mathcal{A}).
\end{array}
\right.
\end{equation*}

    What is the best fractional derivative to use in formulating a problem? And when we want to address only the spatial idea? And the temporal idea? It is not an easy and trivial task to know which is the best fractional derivative than the data of real problems, and which is the best fractional derivative that is chosen to perform an analytical analysis on solutions of differential equations. There are several interesting works involving applications via fractional derivatives that highlight exactly questions about looking at data \cite{ma,ma1,ma2,ma3,ma4}. Over the years with the vast number of definitions of fractional derivatives, it has become more difficult to choose a particular fractional derivative \cite{Kilbas,Miller,Podlubny}. In this sense, in 2018, Sousa and Oliveira \cite{Sousa}, motivated by the HFD and Riemann-Liouville fractional derivative (RLFD) with respect to another function, introduced the so-called $\psi$-Hilfer fractional derivative ($\psi$-HFD), which contains a wide class of particular cases existing fractional derivatives. From this $\psi$-HFD, it allowed and opened doors for numerous open questions, in particular, what is the best choice for the fractional derivative. Furthermore, it has drawn attention to the discussion involving the theory of differential equations. On the other hand, although the $\psi$-HFD allowed and opened up new ranges of options for research, there are still open problems that make it difficult to investigate some high-level issues, in particular the issues we will address in this this paper. Furthermore, it is worth noting that other generalized fractional operators have drawn attention with interesting applications, particularly involving numerical problems \cite{Baleanu,Baleanu2,Baleanu3}. In 2022 Baleanu and Shiri \cite{Baleanu}, discussed a system of fractional differential equations via discretized piecewise polynomial collocation methods to approximate the exact solution, in particular, using the fixed point theorem. For other work that deserves attention in the area of fractional differential equations, see  
    \cite{Baleanu4,Baleanu5,Baleanu6}.

The theory of fractional differential and integro-differential equations is well consolidated and allowed an important advance in the area of differential equations and fractional calculus itself \cite{Kilbas,Miller,Podlubny}. Investigating questions of existence, uniqueness, regularity, attractiveness of solutions of FDEs, has been the subject of research by numerous researchers \cite{Sousa3,Sousa4,Wang,Zhou}. It is remarkable the exponential growth that the theory of FDEs has gained in recent years, we can highlight some works and the references therein \cite{Li1,Liu,Wang10,Sousa11,Sousa10,Larrouy,Norouzi}.

In 1999, Arrieta et al. \cite{Arrieta}, investigated the existence, uniqueness and regularity of solutions for heat equations with nonlinear boundary conditions. In 2016 Andrade and Viana \cite{Andrade5} discussed the local existence, uniqueness, and continuous dependence of mild solutions of an abstract integro-differential equation of the form
\begin{equation*}
    u'=Au+\int_{0}^{t} g(u-s,u(s))ds+f(u,u(t)), \,\,t>0
\end{equation*}
with $u(0)=u_{0}\in D(A)$. In the following year  \cite{Andrade} discussed properties of Volterra integro-differential equations and performed some applications to parabolic models with memory. 

In 2018, Abadias et al. \cite{Abadias} investigated the existence, uniqueness and regularity of a class of abstract nonlinear integral equations of convolution type defined on a Banach space, given by
\begin{equation*}
    u(t)=\int_{0}^{t}a(t-s)[Au(s)+f(s,u(s))] ds,\,\, t\in [0,T],
\end{equation*}
where $A$ is a closed linear operator with dense domain $D(A)$ defined on a Banach space $X$, $u\in L^{1}_{loc}(\mathbb{R})$, $u_{0}\in\mathbb{R}$ and $f:\mathbb{R}_{+}\times X\rightarrow X$.
Other interesting works, can be seen in \cite{Andrade3,Teixeira,Komornik}.

In 2015, Li \cite{Li} investigate the regularity of mild solutions for fractional abstract Cauchy problem with order $\alpha\in (1,2)$, given by
\begin{equation*}
    D^{\alpha}_{t} u(t)+A u(t)= f(t),\,\, t\in (0,T]
\end{equation*}
with $u(0)=u_{0}$, $u'(0)=x_{1}$, where $f:[0,T)\rightarrow X$, $A\in C^{\alpha}(M,\omega)$, $x_{0},x_{1}\in D(A)$, $f\in L^{p}((0,T),X)$ with $p>1$ and $D^{\alpha}_{t}$ is the fractional derivative of order $0<\alpha<1$. Based on properties and analytic solution operator, Li obtain the sufficient condition under which a mild solutions becomes a classical solution, and if the Cauchy problem has an analytic solution operator. Other interesting results, can be seen in \cite{Andrade1,Andrade2,Carvalho20,Djilali,Paulo}.

Let $\alpha$ and $\beta$ be strictly positive real numbers. Then $\mathcal{E}_{\alpha,\beta}:\mathbb{C}\rightarrow\mathbb{C}$ is the Mittag-Leffler function (two parameters), given by \cite{Gorenflo}
\begin{equation}\label{mittagII}
    \mathcal{E}_{\alpha,\beta}(z)=\sum_{k=0}^{\infty} \frac{z^{k}}{\Gamma(\alpha z+\beta)}.
\end{equation}

For $\beta=1$ in Eq. (\ref{mittagII}), we have a one-parameter Mittag-Leffler function given by \cite{Gorenflo}
\begin{equation}\label{mittagI}
    \mathcal{E}_{\alpha}(z)=\sum_{k=0}^{\infty} \frac{z^{k}}{\Gamma(\alpha z+1)}.
\end{equation}

Also, $\beta=\alpha=1$ in Eq. (\ref{mittagII}), we have the special case, that is, exponential function, $\mathcal{E}_{\alpha,\beta}(z)=\mathcal{E}_{1,1}(z)=e^{z}$. Other particular cases for Eq. (\ref{mittagII}), see \cite{Gorenflo}.

Let $J=[a,b]$ be a finite or infinite interval of the line $\mathbb{R}_{+}$ and $0<\alpha \leq 1$. Also, let $\psi(u)$ be an increasing and positive monotone function on $J_1=(a,b]$ having a continuous derivative $\psi'(u)$ on $J_2=(a,b)$. The left-sided fractional integral of a function $\zeta$ with respect to the function $\psi$ on $J=[a,b]$ is defined by \cite{Sousa,Sousa1}
\begin{equation}
I^{\alpha;\psi}_{a+}\zeta(u)=\frac{1}{\Gamma(\alpha)}\int_{a}^{u} \psi'(s)(\psi(u)-\psi(s))^{\alpha-1}\zeta(s)ds.
\end{equation}

Choosing $\psi(u)=u$, we have the RLFI given by
\begin{equation*}
I^{\alpha}_{a+}\zeta(u)=\frac{1}{\Gamma(\alpha)}\int_{a}^{u} (u-s)^{\alpha-1}\zeta(s)ds.
\end{equation*}

On the other hand, let $n-1<\alpha \leq n$ with $n\in \mathbb{N}$, $J$ the interval and $\zeta,\psi\in C^{n}(J,\mathbb{R})$ be two functions such that $\psi$ is increasing and $\psi'(u)\neq 0$ for all $u\in J$. The left-sided $\psi$-HFD $^{\mathbf{H}}\mathbb{D}_{0+}^{\alpha,\beta;\psi}(\cdot)$ of a function $f$ of order $\alpha$ and type $0\leq \beta \leq 1$ is defined by \cite{Sousa,Sousa1}
\begin{equation}
^{H}\mathbb{D}_{a+}^{\alpha,\beta;\psi}\zeta(u)=I^{\beta(n-\alpha);\psi}_{a+} \left(\frac{1}{\psi'(u)} \frac{d}{du}\right)^{n}  I^{(1-\beta)(n-\alpha)}_{a+}\zeta(u).
\end{equation}

Choosing $\psi(u)=u$, we have the HFD, given by
\begin{equation}\label{eq4}
^{H}\mathbb{D}_{a+}^{\alpha,\beta}\zeta(u)=I^{\beta(n-\alpha);\psi}_{a+} \left(\frac{d}{du}\right)^{n}  I^{(1-\beta)(n-\alpha)}_{a+}\zeta(u).
\end{equation}

Motivated by the works discussed above and the open questions, in addition to the restriction of results in the area involving fractional potentials, sectorial operators and interpolation-extrapolation scales of fractional integro-differential equations, in this paper we discuss new results involving $\varepsilon$-regular mild solutions via Gronwall inequality for the problem Eq. (\ref{Eq1}). In order to clarify and facilitate the development of the work, we will now highlight the main contributions, namely:
\begin{enumerate}
    \item We investigate some estimates involving parts of the $\varepsilon$-regular mild solutions of Eq. (\ref{Eq1}) given by means of the two-parameter Mittag-Leffler functions.
    
    \item We discuss the existence, regularity and continuous dependence of $\varepsilon$-regular mild solutions for fractional integro-differential equation (see Eq. (\ref{Eq1})).
    
\end{enumerate}

A natural consequence of problems involving fractional derivatives the particular case $\alpha=1$, is always recovered the classic case. Furthermore, as the HFD is an interpolation between the CFD and RFLD, we have the following fractional versions for Eq. (\ref{Eq1}), i.e.: taking the limit $\beta \rightarrow 0$ in Eq. (\ref{Eq1}), we have the problem in the the RLFD version, given by
\begin{equation*}
\left\{ 
\begin{array}{cll}
^{RL}D_{0+}^{\alpha}\zeta(u) & = & \mathcal{A} \zeta(u)+\displaystyle\int_{0}^{u}g(u-s,\zeta(s))ds+f(u,\zeta(u)),\quad u>0 \\ 
\inte_{0+}^{\alpha}\zeta(0) & = & \zeta_0\in \mathcal{D}(\mathcal{A}).
\end{array}
\right.
\end{equation*}

On the other hand, taking the limit $\beta\rightarrow 1$ in Eq. (\ref{Eq1}), we get the problem in the CFD version, given by
\begin{equation*}
\left\{ 
\begin{array}{cll}
^{C}{\bf D}_{0+}^{\alpha}\zeta(u) & = & \mathcal{A} \zeta(u)+\displaystyle\int_{0}^{u}g(u-s,\zeta(s))ds+f(u,\zeta(u)),\quad u>0 \\ 
\zeta(0) & = & \zeta_0\in \mathcal{D}(\mathcal{A}).
\end{array}
\right.
\end{equation*}

We can notice that the initial conditions of each problem change, consequently, in the formulation of their respective solutions $\varepsilon$-regular mild solutions, changes also occur. Furthermore, all the results investigated here are valid for their respective particular cases.

The rest of the paper is divided as follows: In section 2, we present some essential concepts for the development of the paper, as well as some estimation results involving parts of the $\varepsilon$-regular mild solutions given by the two-parameter Mittag-Leffler function. In section 3, we will attack the main results of this paper, that is, the existence, regularity and continuous dependence of $\varepsilon$-regular mild solutions for fractional integro-differential equations in Banach space. Finally, we close the paper with comments and future work.


\section{Mathematical background: auxiliary results}

Consider $\mathbb{B}_{0}$ a Banach space, $\mathcal{A}: D(\mathcal{A})\subset \mathbb{B}_{0} \rightarrow \mathbb{B}_{0}$ is a linear operator such that $-\mathcal{A}$ is a sectorial operator in $\mathbb{B}_{0}$.
	
Let $\mathbb{B}_{1}=(\mathcal{D}(\mathcal{A}),\left\| \mathcal{A}(\cdot)\right\| )$, and denote $\left\| \cdot\right\|_1 =\left\|\mathcal{A}(\cdot)\right\|$. Let $\mathcal{A}_1$ be the realization of $\mathcal{A}$ on $\mathbb{B}_1$ and define $\mathbb{B}_2=(\mathcal{D}(\mathcal{A}_{1}),\left\| \mathcal{A}_1(\cdot)\right\|_1 )$. Inductively, we define \cite{Amann,Arrieta1}
\begin{align*}
\mathbb{B}_{k+1}=(\mathbb{B}_{k},\left\| \cdot\right\|_{k+1})=\left( \mathcal{D}(\mathcal{A}_{k}),\left\| \mathcal{A}_k(\cdot)\right\| _{k}\right), 
\end{align*}
and $\mathcal{A}_{k+1}= \mathbb{B}_{k+1}$-realization of $\mathcal{A}_k$, for $k\in \nat$.

On the other hand, without loss of generality we assume $0\in \rho(\mathcal{A})$. Thus, consider the space $(\mathbb{B}_{0}, \left\| \mathcal{A}^{-1}\right\| )$ and define $\mathbb{B}_{-1}$ to be the completion of $\mathbb{B}_{0}$ with the norm $\left\| \mathcal{A}^{-1}\right\| $ \cite{Amann,Arrieta1}. It follows that
\begin{align*}
\mathbb{B}_{0}\hookrightarrow \mathbb{B}_{-1}.
\end{align*}

We fix a functor $\mathcal{F}=(\cdot,\cdot)_{\theta}$, $0<\theta<1$. Then, we define
\begin{align*}
\mathbb{B}_{k+\theta}=(\mathbb{B}_{k},\mathbb{B}_{k+1})_{\theta},
\end{align*}
and $\mathcal{A}_{k+\theta}= \mathbb{B}_{k+\theta}$ - realization of $\mathcal{A}_k$, for $k\in \nat\cup\{-1\}$ \cite{Amann,Arrieta1}.

The interpolation-extrapolation scale $[-1,\infty)$ associated with $\mathcal{A} \in \mathcal{F}$, is such that $\{(\mathbb{B}_{\alpha},\mathcal{A}_{\alpha}):-1<\alpha<\infty\}$. An important example of interpolation scale is the fractional power scale of a sectorial operator $\mathcal{A}$.

\begin{theorem}{\rm \cite{Sousa2}}\label{nov} Let $\zeta,v$ be two integrable functions and $g$ continuous, with domain $[a,b]$. Let $\psi \in C^{n}[a,b]$ an increasing function with $\psi'(u)\neq 0$, $\forall u\in [a,b]$. Assume that
\begin{itemize}
    \item $\zeta$ and $v$ are nonnegative;
    
    \item $g$ in nonnegative and nondecreasing.
    
\end{itemize}

If
\begin{equation*}
    \zeta(u)\leq v(u)+g(u)\displaystyle\int_{a}^{u} \psi'(u) (\psi(u)-\psi(\tau))^{\alpha-1}\zeta(\tau) d\tau
\end{equation*}
then,
\begin{equation*}
    \zeta(u)\leq v(u)+\int_{a}^{u} \sum_{k=1}^{\infty} \frac{[g(u) \Gamma(\alpha)]^{k}}{\Gamma(\alpha)} \psi'(\tau) (\psi(u)-\psi(\tau))^{\alpha k-1} v(\tau)d\tau,
\end{equation*}
for all $u\in [a,b]$.
\end{theorem}

\begin{corollary}{\rm \cite{Sousa2}} Under the hypothesis of {\rm {\bf Theorem} \ref{nov}}, let $v$ be a nondecreasing function on $[a,b]$. Then, we have
\begin{equation*}
    \zeta(u)\leq v(u) \mathcal{E}_{\alpha}(g(u) \Gamma(\alpha) [\psi(u)-\psi(\tau)]^{\alpha}),\,\forall u\in [a,b]
\end{equation*}
when $\mathcal{E}_{\alpha}(\cdot)$ is the one-parameter Mittag-Leffler function defined by {\rm Eq. (\ref{mittagI}).}
\end{corollary}

\begin{theorem} {\rm \cite{Carvalho,Andrade10}} Let $\alpha\in(0,1)$ and suppose that $\mathcal{A}:\mathcal{D}(\mathcal{A})\subset \mathbb{B} \rightarrow \mathbb{B}$ is a positive sectorial operator. Then, the operators
	\begin{align*}
	\mitt_\alpha(-u^{\alpha}\mathcal{A}):=\frac{1}{2\pi i}\int_{H_a}e^{\lambda u} \lambda^{\alpha-1}(\lambda^\alpha+\mathcal{A})^{-1}d\lambda,\quad u\geq 0
	\end{align*}
	and
	\begin{align*}
	\mitt_{\alpha,\alpha}(-u^{\alpha}\mathcal{A}):=\frac{u^{1-\alpha}}{2\pi i}\int_{H_a}e^{\lambda u} (\lambda^\alpha+\mathcal{A})^{-1}d\lambda,\quad u\geq 0
	\end{align*}
	are well defined and $\mitt_{\alpha}(-u^{\alpha}\mathcal{A})$ is strongly continuous, i.e., for each $x\in \mathbb{B}$ 
	\begin{align*}
	\lim_{u\rightarrow 0^{+}}\left\| \mitt_\alpha(-u^{\alpha}\mathcal{A})x-x\right\| =0.
	\end{align*}

Furthermore, there exists a constant ${\bf \Theta}>0$ {\rm (uniform to $\alpha$)} such that
	\begin{align*}
	\displaystyle\sup_{u\geq 0}\left\| \mitt_{\alpha}(-u^{\alpha}\mathcal{A})\right\| _{\mathcal{L}(\mathbb{B})}\leq {\bf \Theta},
	\end{align*}
	and
	\begin{align*}
	\displaystyle\sup_{u\geq 0}\left\| \mitt_{\alpha,\alpha}(-u^{\alpha}\mathcal{A})\right\| _{\mathcal{L}(\mathbb{B})}\leq {\bf \Theta}.
	\end{align*}
\end{theorem}

\begin{lemma}{\rm \cite{Carvalho,Andrade10}} Consider $w\in(0,\infty)$, $\zeta:[0,w)\rightarrow \mathbb{B}$ a bounded continuous function, and $f:[0,\infty)\times \mathbb{B}\rightarrow \mathbb{B}$ a continuous function that maps bounded sets onto bounded sets. If $\{u_n\}\subset [0,w)$ satisfies $\lim_{n\rightarrow \infty} u_n=w$, then
	\begin{align*}
	\lim_{n\rightarrow \infty} \int_{0}^{u_n}(u_n-r)^{\alpha-1}\left\|\left(  \mathcal{Q}^{\alpha}_{u_{n},r}(\mathcal{A})-\mathcal{Q}^{\alpha}_{w,s}(\mathcal{A})\right)f(r,\zeta(r))\right\| dr=0,
	\end{align*}
 where $\mathcal{Q}^{\alpha}_{u_{n},r}(\mathcal{A}):= \mitt_{\alpha,\alpha}(-(u_{n}-r)^\alpha \mathcal{A})$ and $\mathcal{Q}^{\alpha}_{w,s}(\mathcal{A}):= \mitt_{\alpha,\alpha}(-(w-s)^\alpha \mathcal{A})$.
\end{lemma}

\begin{theorem}{\rm \cite{Carvalho,Andrade10}}\label{th2.47} Consider $\alpha\in (0,1)$, $0\leq \tilde{\beta} \leq 1$ and suppose that $\mathcal{A}:\mathcal{D}(\mathcal{A})\subset \mathbb{B}\rightarrow \mathbb{B}$ is a positive sectorial operator. Then, there exists a constant ${\bf \Theta}>0$ such that
\begin{align*}
\left\| \mitt_{\alpha}(-u^{\alpha}\mathcal{A})x\right\|_{\mathbb{B}^{\tilde{\beta}}}\leq {\bf \Theta} u^{-\alpha\tilde{\beta}}\left\| x\right\|_{\mathbb{B}_{0}} 
\end{align*}
	and
\begin{align*}
\left\| \mitt_{\alpha,\alpha}(-u^{\alpha}\mathcal{A})x\right\|_{\mathbb{B}_{0}^{\tilde{\beta}}}\leq {\bf \Theta} u^{-\alpha\tilde{\beta}}\left\| x\right\|_{\mathbb{B}_{0}},
\end{align*}
for all $u>0$, where $\mitt_{\alpha}(\cdot)$, $\mitt_{\alpha,\alpha}(\cdot)$ are one and two-parameters Mittag-Leffler functions, respectively.
\end{theorem}
	
Indeed, observe by {\bf Theorem} \ref{th2.47} and some simple computations, that
\begin{align*}
\left\| \mitt_{\alpha}(-u^{\alpha}\mathcal{A})x\right\|_{\mathbb{B}^{1}}\leq {\bf \Theta} u^{-\alpha(1-\tilde{\beta})}\left\| x\right\|_{\mathbb{B}^{\tilde{\beta}}},
\end{align*}
and
\begin{align*}
\left\|u^{\alpha-1} \mitt_{\alpha,\alpha}(-u^{\alpha}\mathcal{A})x\right\|_{\mathbb{B}^{1}}\leq {\bf \Theta} u^{\alpha\tilde{\beta}-1}\left\| x\right\|_{\mathbb{B}^{\tilde{\beta}}},
\end{align*}
for some ${\bf \Theta}>0$.

\begin{remark}{\rm \cite{Carvalho,Andrade10}}\label{remark1} An immediate consequence of {\rm {\bf Theorem} \ref{th2.47}} is that for $0\leq\theta \leq \tilde{\beta}\leq 1$, and $x\in \mathbb{B}^{\tilde{\beta}}$,
\begin{align*}
		u^{\alpha(1+\theta-\tilde{\beta})}\left\| \mitt_{\alpha}(-u^{\alpha}\mathcal{A})x\right\|_{\mathbb{B}^{1+\theta}}\leq {\bf \Theta}\left\|x\right\| _{\mathbb{B}^{\tilde{\beta}}},
\end{align*}
and
\begin{align*}
		u^{\alpha(\theta-\tilde{\beta})+1}\left\|u^{\alpha-1} \mitt_{\alpha,\alpha}(-u^{\alpha}\mathcal{A})x\right\|_{\mathbb{B}^{1+\theta}}\leq {\bf \Theta}\left\|x\right\| _{\mathbb{B}^{\tilde{\beta}}}.
\end{align*}
\end{remark}

\begin{definition}{\rm \cite{Arrieta1}}\label{defi1} For $\ep\geq 0$, we called a function $f$ as $\ep$-regular in relation to $(\mathbb{B}_1,\mathbb{B}_{0})$, if exists $\rho>1$, $\tilde{\gamma}(\ep)$ with $\rho \ep\leq \tilde{\gamma}(\ep)<1$ and a constant $C$, such that $f:\mathbb{B}_{1+\ep}\rightarrow \mathbb{B}_{\tilde{\gamma}(\ep)}$ and 
\begin{align*}
	\left\|f(x)-f(y)\right\| _{\mathbb{B}_{\tilde{\gamma}(\ep)}}\leq C\left\| x-y\right\| _{\mathbb{B}_{1+\ep}}\left(  \left\| x\right\| _{\mathbb{B}_{1+\ep}}^{\rho-1}+ \left\| y\right\| _{\mathbb{B}_{1+\ep}}^{\rho-1}+1\right) ,
	\end{align*}
	for all $x,y\in \mathbb{B}_{1+\ep}$.	 
\end{definition}

Let $\ep,\rho$ (positive constants) and $\tilde{\gamma}(\ep)$, and a function $c\in L_{loc}^{1}(0,\infty)$ with $c(u)\leq Cu^{v}$, for some $C>0$, $\rho\ep-1<v\leq 0$ and $u>0$. Let ${\mathcal V}$ a non-decreasing function, such that $0\leq {\mathcal V(u)}\leq \delta$ and $\lim_{u\rightarrow 0^{+}}{\mathcal V(u)}=0$, for some $\delta>0$.

Let $\mathcal{F}=\mathcal{F}(\ep,\rho,\tilde{\gamma}(\ep),q_1,q_1^{*},c,{\mathcal V},v) $ the class of the functions $f$ with $q_1\in[-v-\tilde{\gamma}(\ep)+\ep,0]$ and $q_1^{*}\in[-v-\tilde{\gamma}(\ep),0]$ such that $f(u,\cdot)$ is a function $\ep$-regular in relation to $(\mathbb{B}_1,\mathbb{B}_{0})$ satisfying 
\begin{align}\label{Eq11}
\left\|f(u,x)-f(u,y)\right\| _{\mathbb{B}_{\tilde{\gamma}(\ep)}}\leq c(u)\left\| x-y\right\| _{\mathbb{B}_{1+\ep}}\left(  \left\| x\right\| _{\mathbb{B}_{1+\ep}}^{\rho-1}+ \left\| y\right\| _{\mathbb{B}_{1+\ep}}^{\rho-1}+{\mathcal V(u)}u^{q_1}\right) ,
\end{align}
and
\begin{align}\label{Eq12}
\left\|f(u,x)\right\|_{\mathbb{B}_{\tilde{\gamma}(\ep)}}\leq c(u)\left(  \left\| x \right\| _{\mathbb{B}_{1+\ep}}^{\rho}+{\mathcal V(u)}u^{q_1^{*}}\right) ,
\end{align}
for all $x,y\in \mathbb{B}_{1+\ep}$ and $u>0$.

First, we introduce the notation  $\mathcal{F}_1={\mathcal{F}(\ep_1,\rho_1,\tilde{\gamma}(\ep_1),q_1,q_1^{*},c,v,{\mathcal V_{1}})}$ and $\mathcal{F}_2={\mathcal{F}(\ep_2,\rho_2,\eta(\ep_2),l,l^{*},c,v,{\mathcal V_{2}})}$, where ${\mathcal V_{1}}$ and ${\mathcal V_{2}}$ are non-decreasing functions.

Consider 
\begin{align*}
\mathcal{B}_{\ep_2}^{\theta,\eta}=\max\left\lbrace \mathcal{B}(\eta(\ep_2)-\theta,2+v+l^{*}),\mathcal{B}(\eta(\ep_2)-\theta,2+v+\rho_2\ep_2),\mathcal{B}(\eta(\ep_2)-\theta,2+v+l-\ep_2)\right\rbrace,
\end{align*}
and
\begin{align*}
\mathcal{B}_{\ep_1}^{\theta,\tilde{\gamma}}=\max\left\lbrace \mathcal{B}(\alpha(\tilde{\gamma}(\ep_1)-\theta),1-\rho_1\ep_1),\mathcal{B}(\alpha(\tilde{\gamma}(\ep_1)-\theta),1+q_1^{*}),\mathcal{B}(\alpha(\tilde{\gamma}(\ep_1)-\theta),1+q_1-\ep_1)\right\rbrace,
\end{align*}
where $\mathcal{B}$ is the beta function.

The function $\zeta:[0,\tau]\rightarrow \mathbb{B}_1$ such that $\zeta\in C((0,\tau],\mathbb{B}_{1+\ep})\cap C([0,\tau],\mathbb{B}_{1})$ and 
\begin{align}\label{solutionmild}
\zeta(u)=\mitt_{\alpha,\beta}(-u^{\alpha}\mathcal{A})\zeta_0+\int_{0}^{u}(u-s)^{\alpha-1}\mathcal{Q}^{\alpha}_{u,s}(\mathcal{A})\int_{0}^{s}g(s-r,\zeta(r))drds+\int_{0}^{u}(u-s)^{\alpha-1}\mathcal{Q}^{\alpha}_{u,s}(\mathcal{A})f(s,\zeta(s))ds,
\end{align}
with $u\geq 0,\tau>0$, is the local mild solutions for the Eq. (\ref{Eq1}) and $\mathcal{Q}^{\alpha}_{u,s}(\mathcal{A}):= \mitt_{\alpha,\alpha}(-(u-s)^\alpha \mathcal{A})$.

Before attacking the main results of this paper, we first investigate some Lemmas, which are of paramount importance throughout the paper.

\begin{lemma}\label{lemma2} Let $f\in \mathcal{F}_1$. If $\zeta\in C((0,\tau],\mathbb{B}_{1+\ep})$, then for all $0\leq \theta<\tilde{\gamma}(\ep_1)$, $\gamma=\alpha+\beta(1-\alpha)$
	\begin{align*}
	u^{\theta\gamma}\left\| \int_{0}^{u}(u-s)^{\alpha-1}\mathcal{Q}^{\alpha}_{u,s}(\mathcal{A})f(s,\zeta(s))ds\right\|_{\mathbb{B}_{1+\theta}}\leq {\bf \Theta \Phi} \mathcal{B}_{\ep_1}^{\theta,\tilde{\gamma}}(\lambda_{\ep_1}(u)^{\rho_1}u^{\theta\gamma+\alpha(\tilde{\gamma}(\ep_1)-\theta)-\rho_1\ep_1}+\mathcal{V}_{1}(u)u^{\theta\gamma+\alpha(\tilde{\gamma}(\ep_1)-\theta)+q_1^{*}}),
	\end{align*}
for all $0<u<\tau$. Here $\lambda_{\ep}=\displaystyle\sup_{s\in[0,\tau]}s^{\ep}\left\|\zeta(s)\right\| _{\mathbb{B}_{1+\ep}} $.
\end{lemma}
\begin{proof} Indeed, by means of the Remark \ref{remark1} and Eq. (\ref{Eq12}), yields
\begin{eqnarray*}
	u^{\theta\gamma}\left\| \int_{0}^{u}(u-s)^{\alpha-1}\mathcal{Q}^{\alpha}_{u,s}(\mathcal{A})f(s,\zeta(s))ds\right\|_{\mathbb{B}_{1+\theta}}&\leq& u^{\theta\gamma} \int_{0}^{u}{\bf \Theta}(u-s)^{\alpha(\tilde{\gamma}(\ep_1)-\theta)-1}\left\| f(s,\zeta(s))\right\|_{\mathbb{B}_{\tilde{\gamma}({\ep_1})}}ds 
	\notag\\ &\leq& u^{\theta\gamma}\Theta \int_{0}^{u}(u-s)^{\alpha(\tilde{\gamma}(\ep_1)-\theta)-1}c(s)\left( \left\|\zeta(s)\right\|_{\mathbb{B}_{1+\ep_1}^{\rho_1}}+{\mathcal V_{1}}(s)s^{q_1^{*}}\right) ds \notag\\ &\leq& {\bf \Theta \Phi} u^{\theta\gamma} \int_{0}^{u}(u-s)^{\alpha(\tilde{\gamma}(\ep_1)-\theta)-1}s^{-\rho_1\ep_1}\left(s^{\ep_1} \left\|\zeta(s)\right\|_{\mathbb{B}_{1+\ep_1}}\right) ^{\rho_1}ds\notag\\ &&+{\bf \Theta \Phi}{\mathcal V_{1}}(s)u^{\theta\gamma}\int_{0}^{u}(u-s)^{\alpha(\tilde{\gamma}(\ep_1)-\theta)-1}s^{q_1^{*}} ds.
\end{eqnarray*}

Introducing the change $s=u\xi$, $ds=ud\xi$, yields
\begin{eqnarray*}
	u^{\theta\gamma}\left\| \int_{0}^{u}(u-s)^{\alpha-1}\mathcal{Q}^{\alpha}_{u,s}(\mathcal{A})f(s,\zeta(s))ds\right\|_{\mathbb{B}_{1+\theta}}\notag&\leq& {\bf \Theta \Phi} u^{\theta\gamma}\lambda_{\ep_1}(u)^{\rho_1} \int_{0}^{1}(u-u\xi)^{\alpha(\tilde{\gamma}(\ep_1)-\theta)-1}(u\xi)^{-\rho_1\ep_1}ud\xi\notag\\
	&&+{\bf \Theta \Phi}{\mathcal V_{1}}(u)u^{\theta\gamma}\int_{0}^{1}(u-u\xi)^{\alpha(\tilde{\gamma}(\ep_1)-\theta)-1}(u\xi)^{q_1^{*}}t d\xi\notag\\&=& {\bf \Theta \Phi} u^{\theta\gamma+\alpha(\tilde{\gamma}(\ep_1)-\theta)-\rho_1\ep_1}\lambda_{\ep_1}(u)^{\rho_1} \int_{0}^{1}(1-\xi)^{\alpha(\tilde{\gamma}(\ep_1)-\theta)-1}\xi^{-\rho_1\ep_1}d\xi\notag\\
	&&+{\bf \Theta \Phi}{\mathcal V_{1}}(u)u^{\theta\gamma+\alpha(\tilde{\gamma}(\ep_1)-\theta)+q_1^{*}}\int_{0}^{1}(1-\xi)^{\alpha(\tilde{\gamma}(\ep_1)-\theta)-1}\xi^{q_1^{*}} d\xi\notag\\
	&=& {\bf \Theta \Phi} u^{\theta\gamma+\alpha(\tilde{\gamma}(\ep_1)-\theta)-\rho_1\ep_1}\lambda_{\ep_1}(u)^{\rho_1}\mathcal{B}\left( \alpha(\tilde{\gamma}(\ep_1)-\theta),1-\rho_1\ep_1\right)  \notag\\
	&&+{\bf \Theta \Phi}{\mathcal V_{1}}(u)u^{\theta\gamma+\alpha(\tilde{\gamma}(\ep_1)-\theta)+q_1^{*}}\mathcal{B}\left( \alpha(\tilde{\gamma}(\ep_1)-\theta),1+q_1^{*}\right)\notag\\
	&\leq& {\bf \Theta \Phi} \mathcal{B}_{\ep_1}^{\theta,\tilde{\gamma}}(\lambda_{\ep_1}(u)^{\rho_1}u^{\theta\gamma+\alpha(\tilde{\gamma}(\ep_1)-\theta)-\rho_1\ep_1}+\mathcal{V}_{1}(u)u^{\theta\gamma+\alpha(\tilde{\gamma}(\ep_1)-\theta)+q_1^{*}}),
\end{eqnarray*}
where, $0<\alpha\leq 1$ and $0\leq\beta\leq 1$.
\end{proof}

Taking $\beta\rightarrow 0$ in the {\bf Lemma} \ref{lemma2}, follows that $\gamma=\alpha$, then we have:
\begin{lemma}\label{lemmaA1}
	Let $f\in \mathcal{F}_1$. If $\zeta\in C((0,\tau],\mathbb{B}_{1+\ep})$, then for all $0\leq \theta<\tilde{\gamma}(\ep_1)$, $0<\alpha\leq 1$
\begin{align*}
	u^{\theta\alpha}\left\| \int_{0}^{u}(u-s)^{\alpha-1}\mathcal{Q}^{\alpha}_{u,s}(\mathcal{A})f(s,\zeta(s))ds\right\|_{\mathbb{B}_{1+\theta}}\leq {\bf \Theta \Phi} \mathcal{B}_{\ep_1}^{\theta,\tilde{\gamma}}(\lambda_{\ep_1}(u)^{\rho_1}u^{\alpha\tilde{\gamma}(\ep_1)-\rho_1\ep_1}+\mathcal{V}_{1}(u)u^{\alpha\tilde{\gamma}(\ep_1)+q_1^{*}},
	\end{align*}
 	for all $0<u<\tau$ and $\lambda_{\ep}=\displaystyle\sup_{s\in[0,\tau]}s^{\ep}\left\|\zeta(s)\right\| _{\mathbb{B}_{1+\ep}} $.
\end{lemma}

On the other hand, taking $\alpha=1$ in the {\bf Lemma} \ref{lemmaA1}, we have the special case, given by:
\begin{lemma} Let $f\in \mathcal{F}_1$. If $\zeta\in C((0,\tau),\mathbb{B}_{1+\ep})$, then for all $0\leq \theta<\tilde{\gamma}(\ep_1)$,
\begin{align*}
	u^{\theta}\left\| \int_{0}^{u}e^{-\mathcal{A}(u-s)}f(s,\zeta(s))ds\right\|_{\mathbb{B}_{1+\theta}}\leq {\bf \Theta \Phi} \mathcal{B}_{\ep_1}^{\theta,\tilde{\gamma}}(\lambda_{\ep_1}(u)^{\rho_1}u^{\tilde{\gamma}(\ep_1)-\rho_1\ep_1}+\mathcal{V}_{1}(u)u^{\tilde{\gamma}(\ep_1)+q_1^{*}},
	\end{align*}
for all $0<u<\tau$ and $\lambda_{\ep}=\displaystyle\sup_{s\in[0,\tau]}s^{\ep}\left\|\zeta(s)\right\| _{\mathbb{B}_{1+\ep}} $.
\end{lemma}

\begin{lemma}\label{lemma3} Let $f\in \mathcal{F}_1$ and $\zeta,\mathcal{V}\in C((0,\tau],\mathbb{B}_{1+\ep_1})$ be such that $u^{\ep_1}\left\|\zeta(u) \right\|_{\mathbb{B}_{1+\ep_{1}}}\leq \mu$ and $u^{\ep_1}\left\|{\mathcal V(u)} \right\|_{\mathbb{B}_{1+\ep_{1}}}\leq \mu$ for some $\mu>0$. Then, for all $0\leq \theta<\tilde{\gamma}(\ep_1)$ and $0<u\leq\tau$, $0\leq\gamma\leq 1$, yields
	\begin{align*}
u^{\theta\gamma}\left\| \int_{0}^{u}(u-s)^{\alpha-1}\mathcal{Q}^{\alpha}_{u,s}(\mathcal{A})\left[ f(s,\zeta(s))-f(s,\mathcal{V}(s))\right]ds\right\|_{\mathbb{B}_{1+\varepsilon_{1}}}\leq \Gamma_{\theta,\gamma}(u)\displaystyle\sup_{s\in[0,\tau]}s^{\ep_1}\left\| \zeta(s)-\mathcal{V}(s)\right\|_{\mathbb{B}_{1+\ep_1}},
	\end{align*}
where $\Gamma_{\theta,\gamma}(u)={\bf \Theta \Phi}\mathcal{B}_{\ep_1}^{\theta,\gamma}u^{\alpha\tilde{\gamma}(\ep_1)}\left( 2\mu^{\rho_1-1}u^{\theta\gamma-\alpha\theta-\rho_1\ep_1}+{\mathcal V_{1}}(u)u^{\theta\gamma-\alpha\theta-{\ep_1 q_1}}\right) $.
\end{lemma}

\begin{proof} Indeed, doing the same procedure in {\bf Lemma} \ref{lemma2}, and using the Eq. (\ref{Eq11}), yields
\begin{align*}
	u^{\theta\gamma}&\left\| \int_{0}^{u}(u-s)^{\alpha-1}\mathcal{Q}^{\alpha}_{u,s}(\mathcal{A})\left[ f(s,\zeta(s))-f(s,\mathcal{V}(s))\right]ds\right\|_{\mathbb{B}_{1+\theta}}\\
	&\leq u^{\theta\gamma} \int_{0}^{u}{\bf \Theta}(u-s)^{\alpha(\tilde{\gamma}(\ep_1)-\theta)-1}\left\|  f(s,\zeta(s))-f(s,\mathcal{V}(s))\right\|_{\mathbb{B}_{\tilde{\gamma}(\ep_1)}}ds \\
	&\leq u^{\theta\gamma} \int_{0}^{u}{\bf \Theta}(u-s)^{\alpha(\tilde{\gamma}(\ep_1)-\theta)-1}c(s))\left\|  \zeta(s)-\mathcal{V}(s)\right\|_{\mathbb{B}_{1+\ep_1}}\left(\left\| \zeta(s)\right\|_{\mathbb{B}_{1+\ep_1}}^{\rho_1-1}+\left\| \mathcal{V}(s)\right\|_{\mathbb{B}_{1+\ep_1}}^{\rho_1-1}+{\mathcal V_{1}}(s)s^{q_1}\right) ds \\
	&\leq {\bf \Theta \Phi}\left\lbrace  u^{\theta\gamma} \int_{0}^{u}(u-s)^{\alpha(\tilde{\gamma}(\ep_1)-\theta)-1}\left(2s^{-\rho_1\ep_1}\mu^{\rho_1-1}+{\mathcal V_{1}}(s)s^{-\ep_1+q_1}\right) ds\right\rbrace  \displaystyle\sup_{s\in[0,\tau]}s^{\ep_1}\left\| \zeta(s)-\mathcal{V}(s)\right\|_{\mathbb{B}_{1+\ep_1}}.
	\end{align*}
	
Introducing the change $s=t\xi$, one has
\begin{eqnarray*}
	&&u^{\theta\gamma}\left\| \int_{0}^{u}(u-s)^{\alpha-1}\mathcal{Q}^{\alpha}_{u,s}(\mathcal{A})\left[ f(s,\zeta(s))-f(s,\mathcal{V}(s))\right]ds\right\|_{\mathbb{B}_{1+\theta}}\\
	&\leq&
	{\bf \Theta \Phi}\left\lbrace  u^{\theta\gamma} \int_{0}^{1}(u-u\xi)^{\alpha(\tilde{\gamma}(\ep_1)-\theta)-1}\left(2(u\xi)^{-\rho_1\ep_1}\mu^{\rho_1-1}+{\mathcal V_{1}}(u)(u\xi)^{-\ep_1+q_1}\right) t d\xi\right\rbrace  \displaystyle\sup_{s\in[0,\tau]}s^{\ep_1}\left\| \zeta(s)-\mathcal{V}(s)\right\|_{\mathbb{B}_{1+\ep_1}}\\
	&=&{\bf \Theta \Phi}\left\lbrace  u^{\theta\gamma+\alpha(\tilde{\gamma}(\ep_1)-\theta)-\rho_1\ep_1} \int_{0}^{1}(1-\xi)^{\alpha(\tilde{\gamma}(\ep_1)-\theta)-1}2\xi^{-\rho_1\ep_1}\mu^{\rho_1-1}d\xi\right. \\
	&&+\left. u^{\theta\gamma+\alpha(\tilde{\gamma}(\ep_1)-\theta)-\ep_1+q_1}{\mathcal V_{1}}(u)\int_{0}^{1}(1-\xi)^{\alpha(\tilde{\gamma}(\ep_1)-\theta)-1}\xi^{-\ep_1+q_1}d\xi\right\rbrace  \displaystyle\sup_{s\in[0,\tau]}s^{\ep_1}\left\| \zeta(s)-\mathcal{V}(s)\right\|_{\mathbb{B}_{1+\ep_1}}\\
	&=&{\bf \Theta \Phi}\left\lbrace 2\mu^{\rho_1-1} u^{\theta\gamma+\alpha(\tilde{\gamma}(\ep_1)-\theta)-\rho_1\ep_1}\mathcal{B}\left( \alpha(\gamma(\ep_1)-\theta),1-\rho_1\ep_1\right)\right.  \\
	&&+\left. {\mathcal V_{1}}(u)u^{\theta\gamma+\alpha(\tilde{\gamma}(\ep_1)-\theta)-\ep_1+q_1}\mathcal{B}\left( \alpha(\gamma(\ep_1)-\theta),1-\ep_1+q_1\right)\right\rbrace  \displaystyle\sup_{s\in[0,\tau]}s^{\ep_1}\left\| \zeta(s)-\mathcal{V}(s)\right\|_{\mathbb{B}_{1+\ep_1}}\\
	&\leq& {\bf \Theta \Phi} \mathcal{B}_{\ep_1}^{\theta,\gamma}\left(2\mu^{\rho_1-1} u^{\theta\gamma+\alpha(\tilde{\gamma}(\ep_1)-\theta)-\rho_1\ep_1}+{\mathcal V_{1}}(u)u^{\theta\gamma+\alpha(\tilde{\gamma}(\ep_1)-\theta)-\ep_1+q_1} \right) \displaystyle\sup_{s\in[0,\tau]}s^{\ep_1}\left\| \zeta(s)-\mathcal{V}(s)\right\|_{\mathbb{B}_{1+\ep_1}}.
	\end{eqnarray*}
\end{proof}

\begin{lemma}\label{lemma4}	Let $g\in \mathcal{F}_2$. If $\zeta\in C((0,\tau],\mathbb{B}_{1+\ep_2})$, then for all $0\leq \theta<\eta(\ep_2)$ and $\gamma=\alpha+\beta(1-\alpha)$,
\begin{eqnarray*}
&&u^{\theta\gamma}\left\| \int_{0}^{u}(u-s)^{\alpha-1}\mathcal{Q}^{\alpha}_{u,s}(\mathcal{A})\int_{0}^{s}g(s-r,\zeta(r))drds\right\|_{\mathbb{B}_{1+\theta}}\notag\\&&\leq {\bf \Theta \Phi}\mathcal{B}_{\ep_2}^{\theta,\eta}\left( \frac{\lambda_{\ep_2}(u)^{\rho_2}}{1+v-\rho_2\ep_2}u^{\theta\gamma+1+v-\rho_2\ep_2}\right.+\left. {\mathcal V_{2}}(u)\mathcal{B}(1+q_{1}^{*},1+v)u^{\theta\gamma+\alpha+v+q_1^{*}+1}\right), 
\end{eqnarray*}
for all $0<u\leq\tau$.
\end{lemma}

\begin{proof} For $g\in \mathcal{F}_2$, we obtain
\begin{align*}
u^{\theta\gamma}&\left\| \int_{0}^{u}(u-s)^{\alpha-1}\mathcal{Q}^{\alpha}_{u,s}(\mathcal{A})\int_{0}^{s}g(s-r,\zeta(r))drds\right\|_{\mathbb{B}_{1+\theta}}\\
&\leq u^{\theta\gamma}\int_{0}^{u}{\bf \Theta}(u-s)^{\alpha(\eta(\ep_2)-\theta)-1}\int_{0}^{s}\left\|g(s-r,\zeta(r))\right\|_{\mathbb{B}_{\eta(\ep_2)}}drds\\
&\leq {\bf \Theta} u^{\theta\gamma}\int_{0}^{u}(u-s)^{\alpha(\eta(\ep_2)-\theta)-1}\int_{0}^{s}c(r)\left( \left\|\zeta(r)\right\|_{\mathbb{B}_{1+\ep_2}}^{\rho_2}+{\mathcal V_{2}}(s-r)(s-r)^{q_1^{*}}\right) drds\\
&\leq {\bf \Theta \Phi} u^{\theta\gamma}\int_{0}^{u}(u-s)^{\alpha(\eta(\ep_2)-\theta)-1}\int_{0}^{s}r^{v}\left(r^{-\rho_2\ep_2}\left( r^{\ep_2} \left\|\zeta(r)\right\|_{\mathbb{B}_{1+\ep_2}}\right) ^{\rho_2}+{\mathcal V_{2}}(s-r)(s-r)^{q_1^{*}}\right) drds\\
&\leq {\bf \Theta \Phi} u^{\theta\gamma}\int_{0}^{u}(u-s)^{\alpha(\eta(\ep_2)-\theta)-1}\left\lbrace \lambda_{\ep_2}(u)^{\rho_2}\int_{0}^{s}r^{v-\rho_2\ep_2}dr+{\mathcal V_{2}}(u)\int_{0}^{s}(s-r)^{q_1^{*}}r^v dr\right\rbrace ds.
\end{align*}

Making the following change $r=s\xi$, $dr=sd\xi$, yields
\begin{eqnarray*}
&&u^{\theta\gamma}\left\| \int_{0}^{u}(u-s)^{\alpha-1}\mathcal{Q}^{\alpha}_{u,s}(\mathcal{A})\int_{0}^{s}g(s-r,\zeta(r))drds\right\|_{\mathbb{B}_{1+\theta}}\notag\\
&\leq& {\bf \Theta \Phi} u^{\theta\gamma}\int_{0}^{u}(u-s)^{\alpha(\eta(\ep_2)-\theta)-1}\left\lbrace \lambda_{\ep_2}(u)^{\rho_2}\frac{s^{v-\rho_2\ep_2+1}}{v-\rho_2\ep_2+1}+{\mathcal V_{2}}(u)\int_{0}^{1}(s-s\xi)^{q_1^{*}}(s\xi)^v sd\xi \right\rbrace ds\notag\\
&=&{\bf \Theta \Phi} u^{\theta\gamma}\int_{0}^{u}(u-s)^{\alpha(\eta(\ep_2)-\theta)-1}\left\lbrace \lambda_{\ep_2}(u)^{\rho_2}\frac{s^{v-\rho_2\ep_2+1}}{v-\rho_2\ep_2+1}+{\mathcal V_{2}}(u)s^{\alpha+v+q_1^{*}+1}\int_{0}^{1}(1-\xi)^{q_1^{*}}\xi^v d\xi \right\rbrace ds\notag\\
&=&{\bf \Theta \Phi} u^{\theta\gamma}\int_{0}^{u}(u-s)^{\alpha(\eta(\ep_2)-\theta)-1}\lambda_{\ep_2}(u)^{\rho_2}\frac{s^{v-\rho_2\ep_2+1}}{v-\rho_2\ep_2+1}ds\notag\\
&&+{\bf \Theta \Phi} u^{\theta\gamma}{\mathcal V_{2}}(u)\int_{0}^{u}(u-s)^{\alpha(\eta(\ep_2)-\theta)-1}s^{\alpha+v+q_1^{*}+1}\mathcal{B}(1+q_{1}^{*},1+v) ds\notag\\
&=&{\bf \Theta \Phi} u^{\theta\gamma}\frac{\lambda_{\ep_2}(u)^{\rho_2}}{v-\rho_2\ep_2+1}\int_{0}^{u}(u-s)^{\alpha(\eta(\ep_2)-\theta)-1}s^{v-\rho_2\ep_2+1}ds\notag\\
&&+{\bf \Theta \Phi} u^{\theta\gamma}{\mathcal V_{2}}(u)\mathcal{B}(1+q_{1}^{*},1+v)\int_{0}^{u}(u-s)^{\alpha(\eta(\ep_2)-\theta)-1}s^{\alpha+v+q_1^{*}+1} ds\notag\\
&=&{\bf \Theta \Phi} u^{\theta\gamma-\alpha(\eta(\ep_2)-\theta)+1}\frac{\lambda_{\ep_2}(u)^{\rho_2}}{v-\rho_2\ep_2+1}\int_{0}^{u}\left( 1-\frac{s}{u}\right) ^{\alpha(\eta(\ep_2)-\theta)-1}s^{v-\rho_2\ep_2+1}ds\notag\\
&&+{\bf \Theta \Phi} u^{\theta\gamma-\alpha(\eta(\ep_2)-\theta)+1}{\mathcal V_{2}}(u)\mathcal{B}(1+q_{1}^{*},1+v)\int_{0}^{u}\left( 1-\frac{s}{u}\right) ^{\alpha(\eta(\ep_2)-\theta)-1}s^{\alpha+v+q_1^{*}+1} ds.
\end{eqnarray*}

Since,
\begin{align*}
u^{\theta\gamma-\alpha(\eta(\ep_2)-\theta)+1}\int_{0}^{u}\left( 1-\frac{s}{u}\right) ^{\alpha(\eta(\ep_2)-\theta)-1}s^{v-\rho_2\ep_2+1}ds\leq u^{\theta\gamma+1+v-\rho_2\ep_2}\mathcal{B}_{\ep_2}^{\theta,\eta},
\end{align*}
and
\begin{align*}
u^{\theta\gamma-\alpha(\eta(\ep_2)-\theta)+1}\int_{0}^{u}\left( 1-\frac{s}{u}\right) ^{\alpha(\eta(\ep_2)-\theta)-1}s^{\alpha+v+q_1^{*}+1} ds\leq u^{\theta\gamma+\alpha+v+q_1^{*}+1}\mathcal{B}_{\ep_2}^{\theta,\eta},
\end{align*}
follows that
\begin{eqnarray*}
&&u^{\theta\gamma}\left\| \int_{0}^{u}(u-s)^{\alpha-1}\mathcal{Q}^{\alpha}_{u,s}(\mathcal{A})\int_{0}^{s}g(s-r,\zeta(r))drds\right\|_{\mathbb{B}_{1+\theta}}\notag\\
&\leq& {\bf \Theta \Phi} \frac{\lambda_{\ep_2}(u)^{\rho_2}}{v-\rho_2\ep_2+1}u^{\theta\gamma+1+v-\rho_2\ep_2}\mathcal{B}_{\ep_2}^{\theta,\eta}+{\bf \Theta \Phi} {\mathcal V_{2}}(u)\mathcal{B}(1+q_{1}^{*},1+v)u^{\theta\gamma+\alpha+v+q_1^{*}+1}\mathcal{B}_{\ep_2}^{\theta,\eta}\\
\notag\\
&=&{\bf \Theta \Phi}\mathcal{B}_{\ep_2}^{\theta,\eta}\left( \frac{\lambda_{\ep_2}(u)^{\rho_2}}{1+v-\rho_2\ep_2}u^{\theta\gamma+1+v-\rho_2\ep_2}+ {\mathcal V_{2}}(u)\mathcal{B}(1+q_{1}^{*},1+v)u^{\theta\gamma+\alpha+v+q_1^{*}+1}\right).
\end{eqnarray*}
\end{proof}

\begin{lemma}\label{lemma5} Let $g\in \mathcal{F}_2$ and $\zeta,\mathcal{V}\in C((0,\tau],\mathbb{B}_{1+\ep_2})$ be such that $u^{\ep_2}\left\| \zeta(u)\right\|_{\mathbb{B}_{1+\ep_2}}\leq \mu $, $u^{\ep_2}\left\| {\mathcal V(u)}\right\|_{\mathbb{B}_{1+\ep_2}}\leq \mu $ for some $\mu>0$. Then, for all $0\leq \theta<\eta(\ep_2)$ and $0<u\leq\tau$, yields
	\begin{align*}
	u^{\theta\gamma}&\left\| \int_{0}^{u}(u-s)^{\alpha-1}\mathcal{Q}^{\alpha}_{u,s}(\mathcal{A})\int_{0}^{s}\left[ g(s-r,\zeta(r))-g(s-r,\mathcal{V}(r))\right] dr ds\right\|_{\mathbb{B}_{1+\theta}}\leq \Gamma_{\theta,\gamma}^2(u) \displaystyle\sup_{s\in[0,\tau]}s^{\ep_2}\left\| \zeta(s)-\mathcal{V}(s)\right\|_{\mathbb{B}_{1+\ep_2}},
	\end{align*}
	where $\Gamma_{\theta,\gamma}^{2}(u)={\bf \Theta \Phi}\mathcal{B}_{\ep_2}^{\theta,\eta}\left(\dfrac{ 2\mu^{\rho_2-1}u^{\theta\gamma+1+v+\alpha(\eta(\ep_2)-\theta)-\rho_2\ep_2}}{1+v-\rho_2\ep_2}+{\mathcal V_{2}}(u)\mathcal{B}(1+q_1,1+v-\ep_2)u^{\theta\gamma+\alpha(\eta(\ep_2)-\theta)+q_1-\ep_2}\right) $.
\end{lemma}

\begin{proof} First, we estimate the following
	\begin{align}\label{eq*}
	\int_{0}^{s}&\left\| g(s-r,\zeta(r))-g(s-r,\mathcal{V}(r))\right\| _{\mathbb{B}_{\eta(\ep_2)}}dr\nonumber\\
	&\leq \int_{0}^{s}c(r)\left\| \zeta(r)-\mathcal{V}(r)\right\|_{\mathbb{B}_{1+\ep_2}}\left( \left\| \zeta(r)\right\| _{\mathbb{B}_{1+\ep_2}}^{\rho_2-1}+\left\| \mathcal{V}(r)\right\| _{\mathbb{B}_{1+\ep_2}}^{\rho_2-1}+{\mathcal V_{2}}(s-r)(s-r)^{q_1}\right) dr \nonumber\\
	&\leq C\left( 2\mu^{\rho_2-1}\int_{0}^{s}r^{v-\rho_2\ep_2}dr+{\mathcal V_{2}}(u)\int_{0}^{s}r^{v-\ep_2}(s-r)^{q_1}dr\right)\displaystyle\sup_{r\in[0,s]}r^{\ep_2}\left\| \zeta(r)-\mathcal{V}(r)\right\|_{\mathbb{B}_{1+\ep_2}}\nonumber\\
	&= C\left( 2\mu^{\rho_2-1}\frac{s^{1+v-\rho_2\ep_2}}{1+v-\rho_2\ep_2}+{\mathcal V_{2}}(u)s^{1+v+q_1-\ep_2}\int_{0}^{1}r^{v-\ep_2}(1-r)^{q_1}dr\right)\displaystyle\sup_{r\in[0,s]}r^{\ep_2}\left\| \zeta(r)-\mathcal{V}(r)\right\|_{\mathbb{B}_{1+\ep_2}}\nonumber\\
	&=Cs^{1+v}\left( \frac{2\mu^{\rho_2-1}s^{-\rho_2\ep_2}}{1+v-\rho_2\ep_2}+{\mathcal V_{2}}(u)\mathcal{B}(1+q_1,1+v-\ep_2)s^{q_1-\ep_2}\right) \displaystyle\sup_{r\in[0,s]}r^{\ep_2}\left\| \zeta(r)-\mathcal{V}(r)\right\|_{\mathbb{B}_{1+\ep_2}}.
	\end{align}

In this sense, using the Eq. (\ref{eq*}) yields
\begin{eqnarray*}
&&u^{\theta\gamma}\left\| \int_{0}^{u}(u-s)^{\alpha-1}\mathcal{Q}^{\alpha}_{u,s}(\mathcal{A})\int_{0}^{s}\left[ g(s-r,\zeta(r))-g(s-r,\mathcal{V}(r))\right] drds\right\|_{\mathbb{B}_{1+\theta}}\\
&\leq& {\bf \Theta}u^{\theta\gamma}\int_{0}^{u}(u-s)^{\alpha(\eta(\ep_2)-\theta)-1}\int_{0}^{s}\left\|  g(s-r,\zeta(r))-g(s-r,\mathcal{V}(r))\right\|_{\mathbb{B}_{1+\theta}}  drds\\
&\leq&\left(  {\bf \Theta}u^{\theta\gamma}\int_{0}^{u}(u-s)^{\alpha(\eta(\ep_2)-\theta)-1}Cs^{1+v}\frac{2\mu^{\rho_2-1}}{1+v-\rho_2\ep_2}s^{-\rho_2\ep_2}ds\right) \displaystyle\sup_{s\in[0,\tau]}s^{\ep_2}\left\| \zeta(s)-\mathcal{V}(s)\right\|_{\mathbb{B}_{1+\ep_2}}\\
&&+\left( {\bf \Theta}u^{\theta\gamma}\int_{0}^{u}(u-s)^{\alpha(\eta(\ep_2)-\theta)-1}{\mathcal V_{2}}(u)\mathcal{B}(1+q_1,1+v-\ep_2)s^{q_1-\ep_2}ds\right)\displaystyle\sup_{s\in[0,\tau]}s^{\ep_2}\left\| \zeta(s)-\mathcal{V}(s)\right\|_{\mathbb{B}_{1+\ep_2}}. 
\end{eqnarray*}

Making the following change $s=u\xi$, then $ds=ud\xi$, one has
\begin{eqnarray*}
&&u^{\theta\gamma}\left\| \int_{0}^{u}(u-s)^{\alpha-1}\mathcal{Q}^{\alpha}_{u,s}(\mathcal{A})\int_{0}^{s}\left[ g(s-r,\zeta(r))-g(s-r,\zeta(r))\right] drds\right\|_{\mathbb{B}_{1+\theta}}\\
&\leq& \left(\frac{2\mu^{\rho_2-1}C{\bf \Theta}u^{\theta\gamma+1+\alpha(\eta(\ep_2)-\theta)+v-\rho_2\ep_2}}{1+v-\rho_2\ep_2}  \int_{0}^{1}(1-\xi)^{\alpha(\eta(\ep_2)-\theta)-1}\xi^{1+v-\rho_2\ep_2}d\xi\right) \displaystyle\sup_{s\in[0,\tau]}s^{\ep_2}\left\| \zeta(s)-\mathcal{V}(s)\right\|_{\mathbb{B}_{1+\ep_2}}\\
&&+\left( {\bf \Theta}u^{\theta\gamma+\alpha(\eta(\ep_2)-\theta)+q_1-\ep_2}{\mathcal V_{2}}(u)\mathcal{B}(1+q_1,1+v-\ep_2)\int_{0}^{1}(1-\xi)^{\alpha(\eta(\ep_2)-\theta)-1}\xi^{q_1-\ep_2}ds\right)\displaystyle\sup_{s\in[0,\tau]}s^{\ep_2}\left\| \zeta(s)-\mathcal{V}(s)\right\|_{\mathbb{B}_{1+\ep_2}}\\
&=&\left(\frac{2\mu^{\rho_2-1}{\bf \Theta \Phi}u^{\theta\gamma+1+\alpha(\eta(\ep_2)-\theta)+v-\rho_2\ep_2}}{1+v-\rho_2\ep_2}\mathcal{B}_{\ep_2}^{\theta,\eta}(\alpha(\eta(\ep_2)-\theta),2+v-\rho_2\ep_2)\right) \displaystyle\sup_{s\in[0,\tau]}s^{\ep_2}\left\| \zeta(s)-\mathcal{V}(s)\right\|_{\mathbb{B}_{1+\ep_2}}\\
&&+\left( {\bf \Theta}u^{\theta\gamma+\alpha(\eta(\ep_2)-\theta)+q_1-\ep_2}{\mathcal V_{2}}(u)\mathcal{B}(1+q_1,1+v-\ep_2)\mathcal{B}_{\ep_2}^{\theta,\eta}(\alpha(\eta(\ep_2)-\theta),q_1-\ep_2+1)\right)\displaystyle\sup_{s\in[0,\tau]}s^{\ep_2}\left\| \zeta(s)-\mathcal{V}(s)\right\|_{\mathbb{B}_{1+\ep_2}}\\
&\leq& {\bf \Theta \Phi}\mathcal{B}_{\ep_2}^{\theta,\eta}\left( \frac{2\mu^{\rho_2-1}}{1+v-\rho_2\ep_2}u^{\theta\gamma+1+\alpha(\eta(\ep_2)-\theta)+v-\rho_2\ep_2}\right)\displaystyle\sup_{s\in[0,\tau]}s^{\ep_2}\left\| \zeta(s)-\mathcal{V}(s)\right\|_{\mathbb{B}_{1+\ep_2}}\\
&&+ {\bf \Theta \Phi} \mathcal{B}_{\ep_2}^{\theta,\eta}\left( {\mathcal V_{2}}(u)\mathcal{B}(1+q_1,1+v-\ep_2)u^{\theta\gamma+\alpha(\eta(\ep_2)-\theta)+q_1-\ep_2}\right) \displaystyle\sup_{s\in[0,\tau]}s^{\ep_2}\left\| \zeta(s)-\mathcal{V}(s)\right\|_{\mathbb{B}_{1+\ep_2}}\\
&=&{\bf \Theta \Phi}\mathcal{B}_{\ep_2}^{\theta,\eta}\left(\frac{2\mu^{\rho_2-1}}{1+v-\rho_2\ep_2}u^{\theta\gamma+1+\alpha(\eta(\ep_2)-\theta)+v-\rho_2\ep_2}+{\mathcal V_{2}}(u)\mathcal{B}(1+q_1,1+v-\ep_2)u^{\theta\gamma+\alpha(\eta(\ep_2)-\theta)+q_1-\ep_2}\right)\\
&&\times\displaystyle\sup_{s\in[0,\tau]}s^{\ep_2}\left\| \zeta(s)-\mathcal{V}(s)\right\|_{\mathbb{B}_{1+\ep_2}},
\end{eqnarray*}
therefore, we concluded the proof.	 
\end{proof}

\section{Main results}

\begin{theorem}\label{theo1} Let $f\in\mathcal{F}_1$ and $g\in\mathcal{F}_2$. If $v_0\in \mathbb{B}_1$ and $\min\{\tilde{\gamma}(\ep_1),\eta(\ep_2)\}>\ep:=\max\{\ep_1,\ep_2\}>0$, then there exist $r>0$ and $\tau_0>0$ such that for any $\zeta_0\in\mathcal{B}_{\mathbb{B}_1}(v_0,r)$ there exists an $\ep$-regular mild solutions of the {\rm Eq. (\ref{Eq1})} which is the unique one satisfying 
\begin{align*}
u^{\ep_i\gamma}\left\| \zeta(u,\zeta_{0})\right\|_{\mathbb{B}_{1+\ep_i}}\xrightarrow{u\rightarrow 0^+} 0,\quad i=1,2.
\end{align*} 

Furthermore, this solution verifies 
\begin{align*}
\zeta\in C((0,\tau_0),\mathbb{B}_{1+\theta}),\quad 0<\theta<min\{\tilde{\gamma}(\ep_1),\eta(\ep_2)\}
\end{align*}
and
\begin{align}\label{14}
u^{\theta \gamma}\left\| \zeta(u,\zeta_{0})\right\|_{\mathbb{B}_{1+\theta}}\xrightarrow{u\rightarrow 0^+} 0<\theta<\min\{\tilde{\gamma}(\ep_1),\eta(\ep_2)\}.
\end{align}

If $\zeta_0,w_0\in\mathcal{B}(v_0,r)$, then 
 \begin{align*}
 u^{\theta \gamma}\left\| \zeta(u,\zeta_{0})-\zeta(u,w_0)\right\|_{\mathbb{B}_{1+\theta}}\leq C\left\| \zeta_0-w_0\right\|_{\mathbb{B}_1},
 \end{align*}
 for all $u\in[0,\tau_0]$, $0\leq \theta<\theta_0<\min\{\tilde{\gamma}(\ep),\eta(\ep)\}$, $\gamma=\alpha+\beta(1-\alpha)$, $0<\alpha\leq 1$ and $0\leq\beta\leq 1$.
\end{theorem}

The {\bf Theorem \ref{theo1}} can be regarded in the following way:

\begin{corollary} Let $f,g$ be $\ep_1$-regular and $\ep_2$-regular maps relative to the pair $(\mathbb{B}_1,\mathbb{B}_{0})$ with $\tilde{\gamma}(\ep_1)$ and $\eta(\ep_2)$, respectively. If $v_0\in \mathbb{B}_1$ and $\min\{\tilde{\gamma}(\ep_1),\eta(\ep_2)\}>\ep:=\max\{\ep_1,\ep_2\}>0$, then there exists $r>0$ and $\tau_0>0$ such that for any $\zeta_0\in\mathcal{B}_{\mathbb{B}_1}(v_0,r)$. There exists an $\ep$-regular mild solutions $\zeta\in C([0,\tau_0],\mathbb{B}_1)\cap C((0,\tau_0],\mathbb{B}_{1+\ep})$ for problem
\begin{equation}\label{15-16}
 \left\{ 
 \begin{array}{cll}
 ^{H}\der_{0+}^{\alpha,\beta}\zeta(u) & = & \mathcal{A}\zeta(u)+\displaystyle\int_{0}^{u}g(\zeta(s))ds+f(\zeta(u)),\quad u>0 \\ 
 \inte_{0+}^{1-\gamma}\zeta(0) & = & \zeta_0,
	 \end{array}
	 \right.
	 \end{equation}
	 which is the unique one satisfying 
 \begin{align*}
	 u^{\ep_i\gamma}\left\| \zeta(u,\zeta_{0})\right\|_{\mathbb{B}_{1+\ep_i}}\xrightarrow{u\rightarrow 0^+} 0,\quad i=1,2.
	 \end{align*}
	 
Furthermore, this solution verifies 
	 \begin{align*}
	 \zeta\in C((0,\tau_0),\mathbb{B}_{1+\theta}),\quad 0<\theta<min\{\tilde{\gamma}(\ep_1),\eta(\ep_2)\}
	 \end{align*}
	 and
	 \begin{align*}
	 u^{\theta \gamma}\left\| \zeta(u,\zeta_{0})\right\|_{\mathbb{B}_{1+\theta}}\xrightarrow{u\rightarrow 0+} 0<\theta<\min\{\tilde{\gamma}(\ep_1),\eta(\ep_2)\}.
	 \end{align*}

Moreover, if $\zeta_0,w_0\in\mathcal{B}(v_0,r)$, then 
	\begin{align*}
	u^{\theta\gamma}\left\| \zeta(u,\zeta_{0})-\zeta(u,w_0)\right\| _{\mathbb{B}_{1+\theta}}\leq C \left\| \zeta_0-w_0\right\|_{\mathbb{B}_1},
	\end{align*}
	for all $u\in[0,\tau_0]$, $0\leq \theta < \theta_0<\min\{\tilde{\gamma}(\ep_1),\eta(\ep_2)\}$.
\end{corollary}

\begin{proof} ({\bf Theorem} \ref{theo1}) In proving this result, we aim to investigate: existence, uniqueness and continuous dependence on initial data of $\ep$-regular mild solutions. In {\bf STEP 1}, first, let's prove the convergence of Eq. (\ref{14}). In this sense, the existence of $\ep$-regular mild solutions will also be investigated.
{\bf STEP 1:} \label{step1} Then, first let $\tilde{\mathcal{B}}=\max\{\mathcal{B}_{\ep_2}^{\ep_2,\eta},\mathcal{B}_{\ep_1}^{\ep_1,\eta},\mathcal{B}_{\ep_1}^{\ep_{2},\tilde{\gamma}},\mathcal{B}_{\ep_1}^{\ep_{1},\tilde{\gamma}},\mathcal{B}(1+l,1+v-\ep_2),\mathcal{B}(1+l^{*},1+v)\}$,
and take $0<\mu\leq 1$, such that
\begin{align}\label{19}
{\bf \Theta \Phi}\tilde{\mathcal{B}}\left( \frac{\mu^{\rho_2-1}}{1+v-\rho_2\ep_2}+\mu^{\rho_2-1}\right) \leq\frac{\mu}{8}.
\end{align}

Also, consider $r=\dfrac{\mu}{4{\bf \Theta}}$ and choose $\tau_0\in(0,1]$ such that, for $v_0$ fixed and $(0,\tau_0]$,
\begin{align*}
u^{\ep_i\gamma}\left\| \mitt_{\alpha,\alpha}(-u^{\alpha}\mathcal{A})\right\|_{\mathbb{B}_{1+\ep_i}}\leq \frac{\mu}{2},
\end{align*}
and
\begin{align*}
r_i(u)<\delta_i,\quad i=1,2,
\end{align*}
with $\delta_1$ and $\delta_2$ satisfying 
\begin{align}\label{20}
{\bf \Theta \Phi}\tilde{\mathcal{B}}\left(\delta_1+\delta_2\mathcal{B}\right) \leq\frac{\mu}{8}.
\end{align}

Consider the following complete metric space
\begin{align*}
k(\tau_0)=\left\lbrace \zeta\in C((0,\tau_0],\mathbb{B}_{1+\ep});\max_{i=1,2}\left\lbrace \displaystyle\sup_{(0,\tau_0]}u^{\ep_i\gamma}\left\| \zeta(u)\right\|_{\mathbb{B}_{1+\ep_i}} \right\rbrace \leq \mu\right\rbrace
\end{align*}
with the norm
\begin{align*}
\left\| \zeta\right\|_{k(\tau_0)}=\max_{i=1,2}\left\lbrace \displaystyle\sup_{(0,\tau_0]}u^{\ep_i\gamma}\left\| \zeta(u)\right\|_{\mathbb{B}_{1+\ep_i}}\right\rbrace.
\end{align*}

Also define the operator
\begin{eqnarray*}
\Lambda \zeta(u)=\mitt_{\alpha,\beta}(-u^{\alpha}\mathcal{A})\zeta_0+\int_{0}^{u}(u-s)^{\alpha-1}\mathcal{Q}^{\alpha}_{u,s}(\mathcal{A})\int_{0}^{s}g(s-r,\zeta(r))drds+\int_{0}^{u}(u-s)^{\alpha-1}\mathcal{Q}^{\alpha}_{u,s}(\mathcal{A})f(s,\zeta(s))ds.
\end{eqnarray*}

{\bf Affirmation 1:} $\Lambda$ is well defined in $k(\tau_0)$.

 If $\zeta\in k(\tau_0)$, $0\leq \theta <\min\{\tilde{\gamma}(\ep),\eta(\ep)\}$ and $0<u_2<u_1\leq \tau_0$, follows 
\begin{eqnarray*}
&&\left\| \Lambda \zeta(u_1)-\Lambda \zeta(u_2)\right\|_{\mathbb{B}_{1+\theta}}\notag\\
&=&\left\| \mitt_{\alpha,\beta}(-u_{1}^{\alpha}\mathcal{A})-\mitt_{\alpha,\beta}(-u_{2}^{\alpha}\mathcal{A})+\int_{0}^{u_1}(u_1-s)^{\alpha-1}\mathcal{Q}^{\alpha}_{u_{1},s}(\mathcal{A})\int_{0}^{s}g(s-r,\zeta(r))drds\right.\notag\\
&&+\int_{0}^{u_1}(u_1-s)^{\alpha-1}\mathcal{Q}^{\alpha}_{u_{1},s}(\mathcal{A})f(s,\zeta(s))ds-\int_{0}^{u_1}(u_1-s)^{\alpha-1}\mathcal{Q}^{\alpha}_{u_{2},s}(\mathcal{A})\int_{0}^{s}g(s-r,\zeta(r))drds\notag\\
&&-\left. \int_{0}^{u_2}(u_2-s)^{\alpha-1}\mathcal{Q}^{\alpha}_{u_{2},s}(\mathcal{A})f(s,\zeta(s))ds \right\|_{\mathbb{B}_{1+\theta}}\notag\\
&\leq& \left\| \left( \mitt_{\alpha,\beta}(-u_1^{\alpha}\mathcal{A})-\mitt_{\alpha,\beta}(-u_2^{\alpha}\mathcal{A})\right) \zeta_0\right\| _{\mathbb{B}_{1+\theta}}\notag\\
&&+\underbrace{\left\| \int_{0}^{u_2} \left\lbrace (u_1-s)^{\alpha-1}\mathcal{Q}^{\alpha}_{u_{1},s}(\mathcal{A})-(u_2-s)^{\alpha-1}\mathcal{Q}^{\alpha}_{u_{2},s}(\mathcal{A})\right\rbrace f(s,\zeta(s))ds\right\|_{\mathbb{B}_{1+\theta}}}_{{\rm(I)}}\notag\\
&&+\underbrace{\left\| \int_{u_1}^{u^2}(u_2-s)^{\alpha-1}\mathcal{Q}^{\alpha}_{u_{1},s}(\mathcal{A})f(s,\zeta(s))ds\right\|_{\mathbb{B}_{1+\theta}}}_{{\rm(II)}}\notag\\
&&+\underbrace{\left\| \int_{0}^{u_2}\left\lbrace (u_1-s)^{\alpha-1}\mathcal{Q}^{\alpha}_{u_{1},s}(\mathcal{A})-(u_2-s)^{\alpha-1}\mathcal{Q}^{\alpha}_{u_{2},s}(\mathcal{A})\right\rbrace \int_{0}^{s}g(s-r,\zeta(r))drds\right\|_{\mathbb{B}_{1+\theta}}}_{{\rm(III)}}\notag\\
&&+\underbrace{\left\| \int_{u_2}^{u_1}(u_1-s)^{\alpha-1}\mathcal{Q}^{\alpha}_{u_{1},s}(\mathcal{A})\int_{0}^{s}g(s-r,\zeta(r))drds\right\|_{\mathbb{B}_{1+\theta}}}_{{\rm(IV)}}.
\end{eqnarray*}

Since $u_2>0$, using the {\bf Theorem} \ref{th2.47}, one has
\begin{eqnarray*}
\left\| \mitt_{\alpha,\beta}(-u_1^{\alpha}\mathcal{A})-\mitt_{\alpha,\beta}(-u_2^{\alpha}\mathcal{A})\zeta_0\right\|_{\mathbb{B}_{1+\theta}}\xrightarrow{u_1\rightarrow u_2} 0.
\end{eqnarray*}

Now, we evaluate (I)-(IV). Note that, the strong continuity of the semigroup in $\mathbb{B}_{1+\theta}$ together with {\bf Lemma} \ref{lemma2} and {\bf Lemma} \ref{lemma4}, follows that the (I) and (III) go to zero as $u_{1}\rightarrow u_{2}$. Now, we want evaluate (II). For this, using the {\bf Lemma} \ref{lemma2}, one has
\begin{eqnarray*}
    \left\| \int_{u_2}^{u_1}(u_1-s)^{\alpha-1}\mathcal{Q}^{\alpha}_{u,s}(\mathcal{A})f(s,\zeta(s))ds\right\|_{\mathbb{B}_{1+\theta}}&\leq& {\bf \Theta \Phi}u_1 u^{\theta\gamma+\alpha(\tilde{\gamma}(\ep_1)-\theta)-\rho_1\ep_1}\mu^{\rho_1}\int_{u_2/u_1}^{1}(1-s)^{\alpha(\tilde{\gamma}(\ep_1)-\theta)-1}s^{-\rho_1\ep_1}ds\notag\\&+&{\bf \Theta \Phi}\delta_1 u^{\theta \gamma+\alpha(\tilde{\gamma}(\ep_1)-\theta)+q_1^{*}}\int_{u_2/u_1}^{1}(1-s)^{\alpha(\tilde{\gamma}(\ep_1)-\theta)-1}s^{q_1^{*}}ds\rightarrow 0
\end{eqnarray*}
as $u_1\rightarrow u_2^{*}$. So, the two term does so (IV). Indeed, using {\bf Lemma} \ref{lemma4}, follows that
\begin{eqnarray*}
	&&\left\| \int_{u_1}^{u_2}(u_1-s)^{\alpha-1}\mathcal{Q}^{\alpha}_{u_{1},s}(\mathcal{A})\int_{0}^{s}g(s-r,\zeta(s))dsdr\right\|_{\mathbb{B}_{1+\theta}}\\
	&\leq& \frac{{\bf \Theta \Phi}u^{\theta\gamma+1-\alpha(\eta(\ep_2)-\theta)}}{1+v-\rho_2\ep_2}\mu^{\rho_2}\int_{u_2/u_1}^{u}(1-s)^{\alpha(\eta(\ep_2)-\theta)-1}s^{1+v-\rho_2\ep_2}ds\\
	&+&{\bf \Theta \Phi}u^{\theta\gamma-\alpha(\eta(\ep_2)-\theta)+1}\delta_2\mathcal{B}(1+q_1^{*},1+v)\int_{u_2/u_1}^{u}(u-s)^{\alpha(\eta(\ep_2)-\theta)-1}s^{\alpha+v+q_1^*+1}ds,
\end{eqnarray*}
converges to zero as $u_1\rightarrow u_2^{+}$. If $\mathcal{B}$ similar when $u_1<u_2$.
	
Using the {\bf Lemma} \ref{lemma2} and {\bf Lemma} \ref{lemma4}, yields
\begin{align*}
	\left\| \int_{0}^{u_2}\left\lbrace (u_1-s)^{\alpha-1}\mathcal{Q}^{\alpha}_{u_{1},s}(\mathcal{A})-(u_2-s)^{\alpha-1}\mathcal{Q}^{\alpha}_{u_{1},s}(\mathcal{A})\right\rbrace \int_{0}^{s}g(s-r,\zeta(r))drds\right\|_{\mathbb{B}_{1+\theta}}\rightarrow 0,
\end{align*}
and
\begin{align*}
	\left\| \int_{0}^{u_2}\left\lbrace (u_1-s)^{\alpha-1}\mathcal{Q}^{\alpha}_{u_{1},s}(\mathcal{A})-(u_2-s)^{\alpha-1}\mathcal{Q}^{\alpha}_{u_{1},s}(\mathcal{A})\right\rbrace f(s,\zeta(s))\right\|_{\mathbb{B}_{1+\theta}}\rightarrow 0,
\end{align*}
as $u_1\rightarrow u_2$.

In order to show that $\Lambda u$ belongs to $k(\tau_0)$, we must show that $\displaystyle\max_{i=1,2}\left\lbrace u^{\ep_i\gamma}\left\| \Lambda \zeta(u)\right\|_{1+\ep_i}\right\rbrace \leq \mu $ for all $u\in(0,\tau_0]$.

So, we use {\bf Lemma} \ref{lemma2}, {\bf Lemma} \ref{lemma4}, Eq. (\ref{19}) and Eq. (\ref{20}), follows
\begin{eqnarray*}
&&u^{\theta\gamma}\left\| \Lambda \zeta(u)\right\|_{\mathbb{B}_{1+\theta}}\notag\\
&\leq& u^{\theta\gamma}\left\| \mitt_{\alpha,\beta}(-u^{\alpha}\mathcal{A})\zeta_0\right\|_{\mathbb{B}_{1+\theta}}\notag\\
&&+u^{\theta\gamma}\left\| \int_{0}^{u}(u-s)^{\alpha-1}\mathcal{Q}^{\alpha}_{u,s}(\mathcal{A})\int_{0}^{s}g(s-r,\zeta(r))drds\right\|_{\mathbb{B}_{1+\theta}}\notag\\
&&+u^{\theta\gamma}\left\| \int_{0}^{u}(u-s)^{\alpha-1}\mathcal{Q}^{\alpha}_{u,s}(\mathcal{A})f(s,\zeta(s))ds\right\|_{\mathbb{B}_{1+\theta}}\notag\\
&\leq&{\bf \Theta} r+u^{\theta\gamma}\left\| \mitt_{\alpha,\beta}(-u^{\alpha}\mathcal{A})\zeta_0\right\|_{\mathbb{B}_{1+\theta}}+{\bf \Theta \Phi}\mathcal{B}_{\ep_1}^{\theta,\tilde{\gamma}}\left( \mu^{\rho_1}u^{\theta\gamma+\alpha(\tilde{\gamma}(\ep_1)-\theta)-\rho_1\ep_1}+\delta_1 u^{\theta\gamma+\alpha(\tilde{\gamma}(\ep_1)-\theta)+q_1^{*}}\right)\notag\\
&&+{\bf \Theta \Phi}\mathcal{B}_{\ep_2}^{\theta,\eta}\left( \frac{\mu^{\rho_2}u^{\theta\gamma+1+v-\rho_2\ep_2}}{1+v-\rho_2\ep_2}+\delta_2 \mathcal{B}(1+q_1^*,1+v)u^{\theta\gamma+\alpha+v+q_1^*+1}\right)\notag\\
&\leq& {\bf \Theta \Phi}\tilde{\mathcal{B}}\left( \mu^{\rho_1}u^{\theta\gamma+\alpha(\tilde{\gamma}(\ep_1)-\theta)-\rho_1\ep_1}+\frac{\mu^{\rho_2}u^{\theta\gamma+1+v-\rho_2\ep_2}}{1+v-\rho_2\ep_2}\right)+{\bf \Theta \Phi}\tilde{\mathcal{B}}\left( \delta_1 u^{\theta\gamma+\alpha(\tilde{\gamma}(\ep_1)-\theta)+q_1^*}\right. \notag\\
&&\left.+\delta_2\mathcal{B}(1+q_1^*,1+v)u^{\theta\gamma+\alpha+v+q_1^*+1}\right)+Mr+u^{\theta\gamma}\left\| \mitt_{\alpha,\beta}(-u^{\alpha}\mathcal{A})v_0\right\|_{\mathbb{B}_{1+\theta}}\notag\\
&\leq& {\bf \Theta \Phi}\tilde{\mathcal{B}}\left( \mu^{\rho_1}+\frac{\mu^{\rho_2}}{1+v-\rho_2\ep_2}\right)+{\bf \Theta \Phi}\tilde{\mathcal{B}}(\delta_1+\delta_2\mathcal{B}(1+q_1^*,1+v))+u^{\theta\gamma}\left\| \mitt_{\alpha,\beta}(-u^{\alpha}\mathcal{A})v_0\right\|_{1+\theta}\notag\\
&\leq& u^{\theta\gamma}\left\| \mitt_{\alpha,\beta}(-u^{\alpha}\mathcal{A})v_0\right\|_{1+\theta} +\frac{\mu}{8}+\frac{\mu}{8}=u^{\theta\gamma}\left\| \mitt_{\alpha,\beta}(-u^{\alpha}\mathcal{A})v_0\right\|_{\mathbb{B}_{1+\theta}}+\frac{\mu}{2}.
\end{eqnarray*}

Then, we have 
\begin{align*}
u^{\ep_i\gamma}\left\| \Lambda \zeta(u)\right\|_{\mathbb{B}_{1+\ep_i}}&\leq u^{\ep_i\gamma}\left\| \mitt_{\alpha,\beta}(-u^\alpha \mathcal{A})v_0\right\|_{X_{1+\theta}}+\frac{\mu}{2}\leq\frac{\mu}{2}+\frac{\mu}{2}=\mu, \quad for \quad i=1,2.
\end{align*}

Therefore, $\Lambda \zeta$ belongs to $k(\tau_0)$. So, {\bf Affirmation 1} is verified. Note that
\begin{align}\label{21}
\Gamma_{\ep_i,\gamma}^2
\leq {\bf \Theta \Phi}\tilde{\mathcal{B}}\left( \frac{2\mu^{\rho_2-1}}{1+v-\rho_2\ep_2}+\delta_2\mathcal{B}\right),
\end{align}
and
\begin{align}\label{22}
\Gamma_{\ep_i,\gamma}(u)\leq {\bf \Theta \Phi}\tilde{\mathcal{B}}\left( 2\mu^{\rho_2-1}+\delta_1\right).
\end{align}

Using the {\bf Lemma} \ref{lemma3} and {\bf Lemma} \ref{lemma5}, yields
\begin{align*}
u^{\ep_i\gamma}\left\| \Lambda \zeta(u)-\Lambda {\mathcal V(u)}\right\|_{\mathbb{B}_{1+\ep_i}}\leq \Gamma_{\ep_i,\gamma}^{2}\displaystyle\sup_{s\in[0,\tau_0]}s^{\ep_2}\left\| \zeta(s)-\mathcal{V}(s)\right\|_{\mathbb{B}_{1+\ep_2}}+\Gamma_{\ep_i,\gamma}(u)\displaystyle\sup_{s\in[0,\tau_0]}s^{\ep_1}\left\| \zeta(s)-\mathcal{V}(s)\right\|_{\mathbb{B}_{1+\ep_1}}.
\end{align*}

In this sense, we have
\begin{align*}
u^{\ep_i\gamma}\left\| \Lambda \zeta(u)-\Lambda {\mathcal V(u)}\right\|_{\mathbb{B}_{1+\ep_i}}\leq \left( \Gamma_{\ep_i,\gamma}^2(u)+\Gamma_{\ep_i,\gamma}(u)\right) \left\| \zeta(s)-\mathcal{V}(s)\right\|_{k(\tau_0)}.
\end{align*}

On the other hand, using the Eq. (\ref{19}), Eq. (\ref{20}), Eq. (\ref{21}) and Eq. (\ref{22}), yields
\begin{align*}
\Gamma_{\ep_i,\gamma}^{2}(u)+\Gamma_{\ep_i,\gamma}\leq 2{\bf \Theta \Phi}\tilde{\mathcal{B}}\left(\frac{ \mu^{\rho_2-1}}{1+v-\rho_2\ep_2}+\mu^{\rho_1-1}\right) + {\bf \Theta \Phi}\tilde{\mathcal{B}}(\delta_1+\mathcal{B} \delta_2)\leq \frac{1}{4}+\frac{\mu}{8}<\frac{1}{2}.
\end{align*}

Therefore,
\begin{align*}
\left\| \Lambda \zeta-\Lambda v\right\|_{k(\tau_0)}<\frac{1}{2}\left\| \zeta-v\right\|_{k(\tau_0)}.
\end{align*}

The Banach fixed point theorem guarantees now the existence of one unique fixed point on $k(\tau_0)$. Let $\zeta(u,\zeta_{0})$ be the fixed point found above. The proof until here shows that $\zeta(\cdot,\zeta_0)$ belongs to $C((0,\tau],\mathbb{B}_{1+\theta})$ for $0\leq \theta <\min\{\tilde{\gamma}(\ep),\eta(\ep)\}$. So it remains to show that $\zeta(\cdot,\zeta_0)\in C((0,\tau],\mathbb{B}_1)$. To this purpose, we first show  Eq. (\ref{14}) from {\bf Lemma} \ref{lemma2} and {\bf Lemma} \ref{lemma4}, yields
\begin{eqnarray*}
&&u^{\theta\gamma}\left\| \zeta(u,\zeta_{0})\right\|_{\mathbb{B}_{1+\theta}}\notag\\
&\leq& u^{\theta\gamma}\left\| \mitt_{\alpha,\beta}(-u^{\alpha}\mathcal{A})\zeta_0\right\|_{\mathbb{B}_{1+\theta}}+{\bf \Theta \Phi}\mathcal{B}_{\ep_1}^{\theta,\tilde{\gamma}}\left( \lambda_{\ep_1}(u)^{\rho_1}u^{\theta\gamma+\alpha(\tilde{\gamma}(\ep_1)-\theta)-\rho_1\ep_1}+{\mathcal V_{1}}(u)u^{\theta\gamma+\alpha(\tilde{\gamma}(\ep_1)-\theta)+q_1^*}\right)\notag\\
&&+{\bf \Theta \Phi}\mathcal{B}_{\ep_1}^{\theta,\eta}\left(\frac{ \lambda_{\ep_2}(u)^{\rho_2}}{1+v-\rho_2\ep_2}u^{\theta\gamma+1+v-\rho_2\ep_2}+{\mathcal V_{2}}(u)u^{\theta\gamma+\alpha+v+q_1^{*}+1}\mathcal{B}(1+q_1^*,1+v)\right)\notag\\
&\leq& u^{\theta\gamma}\left\| \mitt_{\alpha,\beta}(-u^\alpha \mathcal{A})\zeta_0\right\|_{\mathbb{B}_{1+\theta}}+{\bf \Theta \Phi}\mathcal{B}_{\ep_2}^{\theta,\eta}u^{\theta\gamma+\alpha(\tilde{\gamma}(\ep_2)-\theta)}\notag\\
&&\times\left( \displaystyle\sup_{s\in[0,\tau]}s^{\ep_2}\left\| \zeta(s,\zeta_0)\right\|_{\mathbb{B}_{1+\ep_2}}^{\rho_2}\frac{u^{-\rho_2\ep_2}}{1+v-\rho_2\ep_2}+{\mathcal V_{2}}(u)\mathcal{B}(1+q_1^*,1+v)u^{q_1^*+1}\right)\notag\\
&&+{\bf \Theta \Phi}\mathcal{B}_{\ep_1}^{\theta,\tilde{\gamma}}u^{\theta\gamma+\alpha(\tilde{\gamma}(\ep_{1})-\theta)} \left( \displaystyle\sup_{s\in[0,\tau]}s^{\ep_1}\left\| \zeta(s,\zeta_0)\right\|_{\mathbb{B}_{1+\ep_1}}^{\rho_1}u^{-\rho_1\ep_1}+{\mathcal V_{1}}(u)u^{q_1^*}\right).
\end{eqnarray*}

So, for $\theta=\ep_i$, yields
\begin{eqnarray}
&&u^{\ep_i\gamma}\left\| \zeta(u,\zeta_{0})\right\|_{\mathbb{B}_{1+\ep_i}}\notag\\
&\leq& u^{\ep_i\gamma}\left\| \mitt_{\alpha,\beta}(-u^\alpha \mathcal{A})\zeta_0\right\|_{\mathbb{B}_{1+\ep_i}}+{\bf \Theta \Phi}\mathcal{B}_{\ep_2}^{\ep_i,\eta}\left( \displaystyle\sup_{s\in[0,\tau]}s^{\ep_2}\left\| \zeta(s,\zeta_0)\right\|_{\mathbb{B}_{1+\ep_2}}^{\rho_2}\frac{1}{1+v-\rho_2\ep_2}+{\mathcal V_{2}}(u)\mathcal{B}(1+q_1^*,1+v)\right)\notag\\
&&+{\bf \Theta \Phi}\mathcal{B}_{\ep_1}^{\ep_i,\tilde{\gamma}}\left( \displaystyle\sup_{s\in[0,\tau]}s^{\ep_1}\left\| \zeta(s,\zeta_0)\right\|_{\mathbb{B}_{1+\ep_1}}^{\rho_1}+{\mathcal V_{1}}(u)\right)\notag\\
&\leq& u^{\ep_i\tilde{\gamma}}\left\| \mitt_{\alpha,\beta}(-u^\alpha \mathcal{A})\zeta_0\right\|_{\mathbb{B}_{1+\ep_i}}+{\bf \Theta \Phi}\mathcal{B}_{\ep_2}^{\ep_i,\eta}\left( \left\| \zeta(s,\zeta_0)\right\|_{k(\tau_0)}\frac{\mu^{\rho_2-1}}{1+v-\rho_2\ep_2}+{\mathcal V_{2}}(u)\mathcal{B}(1+q_1^*,1+v)\right)\notag\\
&&+{\bf \Theta \Phi}\mathcal{B}_{\ep_1}^{\ep_i,\tilde{\gamma}}\left( \left\| \zeta(s,\zeta_0)\right\|_{k(\tau_0)}\mu^{\rho_1-1}+{\mathcal V_{1}}(u)\right)\notag\\
&\leq& u^{\ep_i\tilde{\gamma}}\left\| \mitt_{\alpha,\beta}(-u^\alpha \mathcal{A})\zeta_0\right\| _{\mathbb{B}_{1+\ep_i}}+\frac{1}{8}\left\| \zeta(u,\zeta_{0})\right\|_{k(\tau_0)}+{\bf \Theta \Phi}\tilde{\mathcal{B}}\left( \frac{{\mathcal V_{2}}(u)}{1+q_1^*}+{\mathcal V_{1}}(u)\right).
\end{eqnarray}

Then, we have
\begin{eqnarray}\label{23}
\dfrac{7}{8}\left\| \zeta(u,\zeta_{0})\right\|_{k(\tau_0)}\leq \max_{i=1,2}\left\lbrace \displaystyle\sup_{u\in(0,\tau_0]}u^{\ep_i\gamma}\left\| \mitt_{\alpha,\beta}(-u^{\alpha}\mathcal{A})\zeta_0\right\|_{\mathbb{B}_{1+\ep_i}}\right\rbrace +{\bf \Theta \Phi}\tilde{\mathcal{B}}({\mathcal V_{2}}(u)\mathcal{B}(1+q_1^*,1+v)+{\mathcal V_{1}}(u))\rightarrow 0,
\end{eqnarray}
as $u\rightarrow 0^+$.

Using Eq. (\ref{23}) (convergence) and the former estimates, then $u^{\theta\gamma}\left\| \zeta(u,\zeta_{0})\right\|_{\mathbb{B}_{1+\theta}}\rightarrow 0$ as $u\rightarrow 0^+$, for $0<\theta<\min\{\tilde{\gamma}(\ep),\eta(\ep)\}$. We finish {\bf STEP 1}, with the proof of
\begin{align}\label{24}
\left\| \zeta(u,\zeta_{0})-\zeta_0\right\|_{\mathbb{B}_1}\rightarrow 0^+,\quad \text{as} \quad u\rightarrow 0^+.
\end{align}

Indeed, choosing $\theta=0$ in {\bf Lemma} \ref{lemma2} and {\bf Lemma} \ref{lemma4}, yields
\begin{eqnarray*}
&&\left\| \zeta(u,\zeta_{0})-\zeta_0\right\|_{\mathbb{B}_1}\notag\\
&\leq& \left\| \mitt_{\alpha,\beta}(-u^\alpha \mathcal{A})\zeta_0-\zeta_0\right\|_{\mathbb{B}_1}+{\bf \Theta \Phi}\mathcal{B}_{\ep_1}^{0,\eta} \left\lbrace \displaystyle\sup_{s\in[0,\tau]}s^{\ep_2}\left\| \zeta(s,\zeta_0)\right\| _{\mathbb{B}_{1+\ep_2}}^{\ep_2}\frac{1}{1+v-\rho_2\ep_2}{\mathcal V_{2}}(u)\mathcal{B}(1+q_1^*,1+v)\right\rbrace \\
&&+{\bf \Theta \Phi}\mathcal{B}_{\ep_1}^{0,\tilde{\gamma}}\left\lbrace \displaystyle\sup_{s\in[0,\tau]}s^{\ep_1}\left\| \zeta(s,\zeta_0)\right\|_{\mathbb{B}_{1+\ep_1}}^{\rho_1}+{\mathcal V_{1}}(u)\right\rbrace \rightarrow 0, 
\end{eqnarray*}
as $u\rightarrow 0^+$ for the Eq. (\ref{14}). Therefore, we have the Eq. (\ref{24}) and we conclude that $\zeta(\cdot, \zeta_0)$ is an $\ep$-regular mild solutions starting at $\zeta_0$, in the set $k(\tau_0)$.

{\bf STEP 2.} We are going to show that this solution is the any one in the class of the functions $\phi$ such that \begin{align}\label{25}
u^{\ep_i\gamma}\left\| \phi(u)\right\|_{\mathbb{B}_1+\ep_i}\rightarrow 0^+,\quad \text{as} \quad u\rightarrow 0^+, \quad i=1,2.
\end{align}

If $\phi(u)$ is an $\ep$-regular mild solutions satisfying the Eq. (\ref{25}), $\exists\tilde{\tau}$ with $\tilde{\tau}\leq \tau_0$ such that $\left\| \phi(u)\right\|_{k(\tau_0)}\leq\mu$ on $(0,\tilde{\tau}]$. Note that using {\bf STEP 1}, the unique solution in $k(\tilde{\tau})$ is the solution in $k(\tau_0)$ constrained to $[0,\tilde{\tau}]$. In this sense, we get 
\begin{align*}
\zeta(u,\zeta_{0})=\phi(u),\quad u\in[0,\tilde{\tau}].
\end{align*}

{\bf STEP 3.} Continuous dependency on the initial data.

Indeed, we take advantages of {\bf Lemma} \ref{lemma3} and {\bf Lemma} \ref{lemma5}, to obtain
\begin{eqnarray*}
u^{\theta\gamma}\left\| \zeta(u,\zeta_{0})-\zeta(u,w_0)\right\|_{\mathbb{B}_{1+\theta}}&\leq& {\bf \Theta} \left\| \zeta_0-w_0\right\|_{\mathbb{B}_1}+\Gamma_{\theta}^{2}(u)\displaystyle\sup_{s\in[0,\tau_{0}]}s^{\ep_2}\left\| \zeta(s,\zeta_0)-\zeta(s,w_0)\right\|_{\mathbb{B}_{1+\ep_2}}\\
&&+\Gamma_{\theta}(u)\displaystyle\sup_{s\in[0,\tau]}s^{\ep_1}\left\|\zeta(s,\zeta_0)-\zeta(s,w_0)\right\|_{\mathbb{B}_{1+\ep_1}}.
\end{eqnarray*}

For $i=1,2$, we have been in {\bf STEP 1} that $\Gamma_{\ep_i}(u)+\gamma_{\ep_i}^2(u)\leq\frac{1}{2}$.

Then,
\begin{align*}
u^{\ep_i\gamma}\left\| \zeta(u,\zeta_{0})-\zeta(u,w_0)\right\|_{\mathbb{B}_{1+\ep_i}}&\leq {\bf \Theta}\left\| \zeta_0-w_0\right\|_{\mathbb{B}_1}+\frac{1}{2}\left\|\zeta(u,\zeta_{0})-\zeta(u,w_0)\right\|_{k(\tau_0)}\leq 2{\bf \Theta} \left\| \zeta_0-w_0\right\|_{\mathbb{B}_1}.
\end{align*}

Taking $0\leq \theta \leq \theta_0<\min\{\tilde{\gamma}(\ep),\eta(\ep)\}$, these inequalities imply that
\begin{align*}
u^{\theta\gamma}\left\| \zeta(u,\zeta_{0})-\zeta(u,w_0)\right\|_{\mathbb{B}_{1+\theta}}\leq {\bf \Theta} \left\|\zeta_0-w_0\right\|_{\mathbb{B}_1}+\left(\Gamma_{\theta,\gamma}(u)+\Gamma_{\theta,\gamma}^{2}(u)\right) 2{\bf \Theta}\left\|\zeta_0-w_0\right\|_{\mathbb{B}_1},
\end{align*}
that is,
\begin{align*}
u^{\theta\gamma}\left\| \zeta(u,\zeta_{0})-\zeta(u,w_0)\right\|_{\mathbb{B}_{1+\theta}}\leq c(\theta_0,\tau_0)\left\| \zeta_0-w_0\right\|_{\mathbb{B}_1},
\end{align*}
 where $C(\theta_0,\tau_{0})={\bf \Theta}\left( 1+2\sup\{\Gamma_{\theta,\gamma}(u)+\Gamma_{\theta,\gamma}^2(u), u\in[0,\tau_0],0\leq\theta\leq\theta_0\}\right) $.
\end{proof}

\begin{lemma}\label{lemma6} Set $T_1,T_2\leq 1$. Let $\phi$ and $\psi$ be $\ep$-regular mild solutions of {\rm Eq. (\ref{Eq1})} in $[0,T_1]$ and $[0,T_2]$, respectively, such that they coincide in $[0,\tilde{T}]$, for some $\tilde{T}\leq\min\{T_1,T_2\}$. If $f\in\mathcal{F}_1$ and $g\in\mathcal{F}_2$, then $\phi(u)=\psi(u)$, for all $u\in[0,\min\{T_1,T_2\}]$.
\end{lemma}

\begin{proof} Using {\bf Definition \ref{defi1}} for $\tilde{T}\leq u\leq \min\{T_1,T_2\}$,
\begin{eqnarray}\label{I}
&&\left\| \phi(u)-\psi(u)\right\|_{\mathbb{B}_{1+\ep}}\notag\\
&\leq&\int_{0}^{u}\left\|(u-s)^{\alpha-1}\mathcal{Q}^{\alpha}_{u,s}(\mathcal{A})\right\|_{\mathbb{B}_{1+\ep}}\int_{0}^{s}\left\| g(s-r,\phi(r))-g(s-r,\psi(r))\right\| _{\mathbb{B}_{1+\ep}}drds\notag\\
&&+\int_{0}^{u}\left\| (u-s)^{\alpha-1}\mathcal{Q}^{\alpha}_{u,s}(\mathcal{A})\right\|_{\mathbb{B}_{1+\ep}}\left\| f(s,\phi(s))-f(s,\psi(s))\right\|_{\mathbb{B}_{1+\ep}}ds\notag\\
&\leq& \int_{0}^{u} \left\| (u-s)^{\alpha-1}\mathcal{Q}^{\alpha}_{u,s}(\mathcal{A})\right\|_{\mathbb{B}_{1+\ep}}\int_{0}^{s}c(s-r)r^v\left\| \phi(r)-\psi(r)\right\|_{\mathbb{B}_{1+\ep_2}}\notag\\
&&\times\left( \left\| \phi(r)\right\| _{\mathbb{B}_{1+\ep_2}}^{\rho_2-1}+\left\| \psi(r)\right\| _{\mathbb{B}_{1+\ep_2}}^{\rho_2-1}+{\mathcal V_{2}}(s-r)(s-r)^v \right) drds \notag\\
&&+ \int_{0}^{u} \left\| (u-s)^{\alpha-1}\mathcal{Q}^{\alpha}_{u,s}(\mathcal{A})\right\|_{X_{1+\ep}}c(s)\left\| \phi(s)-\psi(s)\right\|_{\mathbb{B}_{1+\ep_1}}\left( \left\| \phi(s)\right\| _{\mathbb{B}_{1+\ep_1}}^{\rho_1-1}+\left\| \psi(s)\right\| _{\mathbb{B}_{1+\ep_1}}^{\rho_1-1}+{\mathcal V_{1}}(s)s^{q_1}
\right) ds \notag\\
&\leq& {\bf \Theta \Phi} \int_{0}^{u}(u-s)^{\alpha(\eta(\ep_2)-\ep)-1}\int_{0}^{s}\left\| \phi(r)-\psi(r)\right\|_{\mathbb{B}_{1+\ep_2}}r^v\left( \left\| \phi(r)\right\|_{\mathbb{B}_{1+\ep_2}}^{\rho_2-1}+\left\| \psi(r)\right\|_{\mathbb{B}_{1+\ep_2}}^{\rho_2-1}+{\mathcal V_{2}}(s-r)(s-r)^{v}\right)drds\notag\\
&&+{\bf \Theta \Phi} \int_{0}^{u}(u-s)^{\alpha(\tilde{\gamma}(\ep_1)-\ep)-1}\left\| \phi(s)-\psi(s)\right\|_{\mathbb{B}_{1+\ep_1}}\left( \left\| \phi(s)\right\|_{\mathbb{B}_{1+\ep_1}}^{\rho_1-1}+\left\| \psi(s)\right\|_{\mathbb{B}_{1+\ep_1}}^{\rho_1-1}+{\mathcal V_{1}}(s)s^{q_1}\right)ds\notag\\
&=&{\bf \Theta \Phi} \int_{\tilde{T}}^{u}(u-s)^{\alpha(\eta(\ep_2)-\ep)-1}\int_{\tilde{T}}^{s}\left\| \phi(r)-\psi(r)\right\|_{\mathbb{B}_{1+\ep_2}}r^v\left( \left\| \phi(r)\right\|_{\mathbb{B}_{1+\ep_2}}^{\rho_2-1}+\left\| \psi(r)\right\|_{\mathbb{B}_{1+\ep_2}}^{\rho_2-1}+{\mathcal V_{2}}(s-r)^{v}\right)drds\notag\\
&&+{\bf \Theta \Phi} \int_{\tilde{T}}^{u}(u-s)^{\alpha(\tilde{\gamma}(\ep_1)-\ep)-1}\left\| \phi(s)-\psi(s)\right\|_{\mathbb{B}_{1+\ep_1}}\left( \left\| \phi(s)\right\|_{\mathbb{B}_{1+\ep_1}}^{\rho_1-1}+\left\| \psi(s)\right\|_{\mathbb{B}_{1+\ep_1}}^{\rho_1-1}+{\mathcal V_{1}}(s)s^{q_1}\right)ds\notag\\
&\leq& {\bf \Theta \Phi} \left( \displaystyle\sup_{s\in[\tilde{T},u]}\left\lbrace \left\| \phi(s)\right\|_{\mathbb{B}_{1+\ep_2}}^{\rho_2-1}+\left\| \psi(s)\right\|_{\mathbb{B}_{1+\ep_2}}^{\rho_2-1}\right\rbrace \right)\left( \int_{\tilde{T}}^{u}(u-s)^{\alpha(\eta(\ep_2)-\ep)-1}\int_{\tilde{T}}^{s}r^v\left\| \phi(r)-\psi(r)\right\|_{\mathbb{B}_{1+\ep_2}}drds\right)\notag\\
&&+{\bf \Theta \Phi}\delta_2\left( \int_{\tilde{T}}^{u}(u-s)^{\alpha(\eta(\ep_2)-\ep)-1}\int_{\tilde{T}}^{s}\left\| \phi(r)-\psi(r)\right\| _{\mathbb{B}_{1+\ep_2}}(s-r)^v r^v drds\right)\notag\\
&&+{\bf \Theta \Phi}\left( \displaystyle\sup_{s\in[\tilde{T},u]}\left\lbrace \left\|\phi(s)\right\|_{\mathbb{B}_{1+\ep_1}}^{\rho_1-1}+\left\| \psi(s)\right\|_{\mathbb{B}_{1+\ep_1}}^{\rho_1-1}\right\rbrace +\delta_1 \tilde{T}^{q_1}\right)\left( \int_{\tilde{T}}^{u}(u-s)^{\alpha(\tilde{\gamma}(\ep_1)-\ep)-1}\left\| \phi(s)-\psi(s)\right\|_{\mathbb{B}_{1+\ep_1}}ds\right).
\end{eqnarray}

Since $\ep_i\leq\ep$, yields
\begin{align*}
\left\|z \right\|_{\mathbb{B}_{1+\ep_i}}\leq C(\ep_1,\ep_2)\left\|z \right\|_{\mathbb{B}_{1+\ep}},\quad\text{for all}\quad z\in \mathbb{B}_{1+\ep_i}.
\end{align*}

Then, 
\begin{align*}
\left\| \phi(s)-\psi(s)\right\|_{1+\ep_i}\leq C(\ep_1,\ep_2)\left\| \phi(s)-\psi(s)\right\|_{1+\ep},\quad i=1,2.
\end{align*}

Also, the numbers $\displaystyle\sup_{s\in[\tilde{T},u]}\left\| \phi(s)\right\|_{\mathbb{B}_{1+\ep_i}}$ and $\displaystyle\sup_{[\tilde{T},u]}\left\| \psi(s)\right\|_{\mathbb{B}_{1+\ep_i}}$ are well defined $\left(\phi,\psi\in C((0,\min\{T_1,T_2\}]\mathbb{B}_{-1}\right)$, then from Eq. (\ref{I}), yields
\begin{eqnarray*}
\left\| \phi(u)-\psi(u)\right\|_{\mathbb{B}_{1+\ep}}&\leq& {\bf \Theta \Phi}\left(c(\ep_1,\ep_2)\displaystyle\sup_{[\tilde{T},u]}\left\lbrace \left\|\phi(s)\right\|_{\mathbb{B}_{1+\ep}}^{\rho_2-1}+\left\| \psi(s)\right\|_{\mathbb{B}_{1+\ep}}^{\rho_2-1}\right\rbrace \right)\notag\\
&&\times \left(\int_{\tilde{T}}^{u}(u-s)^{\alpha(\eta(\ep_2)-\ep)-1}c(\ep_1,\ep_2)\int_{\tilde{T}}^{s}r^v\left\| \phi(r)-\psi(r)\right\|_{\mathbb{B}_{1+\ep}}drds\right)\notag\\
&&+{\bf \Theta \Phi}\delta_2\left( \int_{\tilde{T}}^{u}(u-s)^{\alpha(\eta(\ep_2)-\ep)-1}\int_{\tilde{T}}^{s}\left\|\phi(r)-\psi(r)\right\|_{\mathbb{B}_{1+\ep}}(s-r)^2r^vdrds\right)\notag\\
&&+{\bf \Theta \Phi}\left( c(\ep_1,\ep_2)\displaystyle\sup_{s\in[\tilde{T},u]}\left\lbrace \left\|\phi(s)\right\|_{\mathbb{B}_{1+\ep}}^{\rho_1-1}+\left\|\psi(s)\right\|_{1+\ep}^{\rho_1-1}\right\rbrace +\delta_1 \tilde{T}^{q_1}\right)\notag\\
&&\times\left( c(\ep_1,\ep_2)\int_{\tilde{T}}^{u}(u-s)^{\alpha(\tilde{\gamma}(\ep_1)-\ep)-1}\left\| \phi(s)-\psi(s)\right\| _{\mathbb{B}_{1+\ep}}ds\right).
\end{eqnarray*}

Note that,
\begin{eqnarray*}
\int_{\tilde{T}}^{s}r^v\left\| \phi(r)-\psi(r)\right\|_{\mathbb{B}_{1+\ep}}dr&\leq& \displaystyle\sup_{r\in[\tilde{T},s]}\left\|\phi(r)-\psi(r)\right\|_{\mathbb{B}_{1+\ep}}\int_{\tilde{T}}^{s}r^vdr\notag\\
&=&\displaystyle\sup_{r\in[\tilde{T},s]} \left\| \phi(r)-\psi(r)\right\| _{\mathbb{B}_{1+\ep}}\frac{(s-\tilde{T})^v}{v+1}\notag\\
&\leq& \frac{1}{v+1}\displaystyle\sup_{r\in[\tilde{T},s]}\left\|\phi(r)-\psi(r)\right\|_{\mathbb{B}_{1+\ep}},
\end{eqnarray*}
and
\begin{align*}
\int_{\tilde{T}}^{s}\left\|\phi(r)-\psi(r)\right\|_{\mathbb{B}_{1+\ep}}(s-r)^{q_1}r^vdr\leq \displaystyle\sup_{r\in[\tilde{T},s]}\left\| \phi(r)-\psi(r)\right\|_{\mathbb{B}_{1+\ep_2}}\int_{\tilde{T}}^{s}(s-r)^{q_1}r^vdr.
\end{align*}

Taking the following variable change $r=st$, we get
\begin{align*}
\int_{\tilde{T}}^{s}\left\| \phi(r)-\psi(r)\right\|_{\mathbb{B}_{1+\ep}}(s-r)^{q_1}r^vdr\leq \displaystyle\sup_{r\in[\tilde{T},s]}\left\|\phi(r)-\psi(r)\right\|_{\mathbb{B}_{1+\ep_2}}\mathcal{B}(1+q_1,1+v).
\end{align*}

Consider $S=\max\left\lbrace \mathcal{A}_1,\dfrac{\mathcal{A}_2}{1+v},c(\ep_1,\ep_2)\delta_2\mathcal{B}(1+q_1,1+v)\right\rbrace $, where
\begin{align*}
\mathcal{A}_2=c(\ep_1,\ep_2)\displaystyle\sup_{s\in[\tilde{T},u]}\left\lbrace \left\| \phi(s)\right\|_{\mathbb{B}_{1+\ep}}^{\rho_2-1}+\left\| \psi(s)\right\|_{\mathbb{B}_{1+\ep}}^{\rho_2-1}\right\rbrace, \quad \mathcal{A}_1=c(\ep_1,\ep_2)\displaystyle\sup_{s\in[\tilde{T},u]}\left\lbrace \left\| \phi(s)\right\|_{\mathbb{B}_{1+\ep}}^{\rho_1-1}+\left\| \psi(s)\right\|_{\mathbb{B}_{1+\ep}}^{\rho_1-1}\right\rbrace +\delta_1 T^{q_1},
\end{align*}
and
$\xi(u)=\displaystyle\sup_{s\in[\tilde{T},u]}\left\| \phi(s)-\psi(s)\right\|_{\mathbb{B}_{1+\ep}} $.

These notes lead to,
\begin{align*}
\left\|\phi(u)-\psi(u)\right|_{\mathbb{B}_{1+\ep}}\leq {\bf \Theta \Phi}S\left\lbrace 2\int_{\tilde{T}}^{u}(u-s)^{\alpha(\eta(\ep_2)-\ep)-1}\xi(s)ds+\int_{\tilde{T}}^{u}(u-s)^{\alpha(\tilde{\gamma}(\ep_2)-\ep)-1}\xi(s)ds\right\rbrace 
\end{align*}

Recalling that $\min\{T_1,T_2\}\leq 1$ and $(u-s)\in[0,1]$, put $k=\min\{\tilde{\gamma}(\ep_1),\eta(\ep_2)\}$. Then, yields
\begin{align*}
(u-s)^{\alpha(\eta(\ep_2)-\ep)-1}\leq (u-s)^{\alpha(k-\ep)-1},\quad\text{and}\quad (u-s)^{\alpha(\tilde{\gamma}(\ep_1)-\ep)-1}\leq (u-s)^{\alpha(k-\ep)-1}. 
\end{align*}
Therefore,
\begin{align*}
\xi(u)\leq 3{\bf \Theta \Phi}S\left( \int_{\tilde{T}}^{u}(u-s)^{\alpha(k-\ep)-1}\xi(s)ds\right).
\end{align*}

Now, from Gronwall inequality, follows that $\xi(u)=0$ and therefore $\phi(u)=\psi(u)$ also for $u\in[\tilde{T}, \min\{T_1,T_0\}]$. Thus, {\bf Lemma} \ref{lemma6} infers that $\phi(u)=\zeta(u,\zeta_{0})$ for all $u\in[0,\tau_0]$ and the uniqueness is proved.
\end{proof}

\section{Examples}

In this section, we present two examples that are related to the problem (\ref{15-16}) and consequently, we are going to use the {\rm \bf Theorem \ref{theo1}}.

\begin{example} Consider $\gamma=1$, $g(u-s,\phi(s))=\dfrac{(u-s)^{\alpha-1}}{\Gamma(\alpha)}$ and $f(u,\phi(u))=0$ in the problem (\ref{15-16}). Note that 
\begin{eqnarray*}
    \int_{0}^{u} g(u-s,\phi(s)) ds=\int_{0}^{u}\frac{(u-s)^{\alpha-1}}{\Gamma(\alpha)}ds= \frac{u^{\alpha}}{\Gamma(\alpha+1)}.
\end{eqnarray*}

Thus, we have the following particular problem given by
\begin{equation}\label{novo21}
 \left\{ 
 \begin{array}{cll}
 ^{H}\der_{0+}^{\alpha,\beta}\zeta(u) & = & \mathcal{A}\zeta(u)+\dfrac{u^{\alpha}}{\Gamma(\alpha+1)},\quad u>0 \\ 
 \zeta(0) & = & \zeta_0,
	 \end{array}
	 \right.
	 \end{equation}
which is the unique one satisfying 
 \begin{align*}
	 u^{\ep_i\gamma}\left\| \zeta(u,\zeta_{0})\right\|_{\mathbb{B}_{1+\ep_i}}\xrightarrow{u\rightarrow 0^+} 0,\quad i=1,2.
	 \end{align*}
\end{example}

Note that the functions $f$ and $g$, satisfy the conditions of Theorem \ref{theo1}. In this sense, we have that the problem (\ref{novo21}) admits the existence and uniqueness of solutions.

\begin{example} Consider the following fractional integro-differential equation
\begin{equation}\label{application1}
 \left\{ 
 \begin{array}{rll}
\dfrac{\partial^{\alpha}}{\partial u^{\alpha}} \zeta(x,u) -\zeta_{xx}(x,u) & = & \psi(u,\zeta(x,u))+ \displaystyle\int_{0}^{u} \rho(s-u) \varphi(s,\zeta(x,s))ds,\,\,\,\, \,\, \,\,\,  \\ 
 \zeta(u,0) & = & \zeta(u,\pi)=0,\,\,u>0\\
 \zeta(0,x)&=&\zeta_{0}(x),\,\,x\in [0,\pi] 
	 \end{array}
	 \right.
\end{equation}
where $\dfrac{\partial^{\alpha}}{\partial u^{\alpha}}(\cdot)$ is the Riemann-Liouville fractional derivative of order $0<\alpha\leq 1$. Let us take $\mathbb{B}_{0}=L^{2}[0,\pi]$. Define the operator $\mathcal{A}: D(\mathcal{A})\subset \mathbb{B}_{0}\rightarrow \mathbb{B}_{0}$ by $\mathcal{A} \zeta= \zeta_{xx}$ where the domain $\mathcal{D}(\mathcal{A})$ is given by $\left\{\zeta\in \mathbb{B}_{0}: \zeta,\zeta_{x}\,\,{\rm are}\,\,{\rm absolulety}\,\,{\rm continuous} ,\,\, \zeta_{xx}\in\mathbb{B}_{0},\zeta(0)=\zeta(\pi)=0\right\}$.

Consider the operators $f(u,\zeta(u))(x)=\psi(u,\zeta(x,u))$ and $g(u-s,\zeta(s))(x)=\rho(s-u)\varphi(s,\zeta(x,s))$ and satisfying the conditions of {\rm \bf Theorem \ref{theo1}}. Then, the above problem can be written was
\begin{equation}\label{application2}
 \left\{ 
 \begin{array}{rll}
{\bf D}_{0+}^{\alpha} \zeta(u) & = & \mathcal{A} \zeta(u)+ \displaystyle\int_{0}^{u} g(s-u,\zeta(s))ds+f(u,\zeta(u)),\,\,u>0\\ 
 \zeta(0) & = & \zeta_0.
 \end{array}
	 \right.
\end{equation}

Note that, taking $\beta=1$ and $\gamma=1$ in problem {\rm(\ref{15-16})}, we have the problem {\rm(\ref{application2})}. In this sense, the problem {\rm(\ref{application1})} and {\rm(\ref{application2})}, are particular cases of the problem {\rm(\ref{15-16})}. Again, we can apply the {\rm \bf Theorem \ref{theo1}}, and guarantee the existence and uniqueness of solutions for the example
{\rm (\ref{application1})}.
\end{example}

\section{Concluding remarks and future works}

In this present work we managed to obtain some interesting results that contribute to the theory of fractional differential equations and applications, which are better characterized as follows:
\begin{itemize}
    \item At first we get some results through lemmas that help in the main results. 

    \item We investigated the existence, regularity and continuous dependence of $\varepsilon$-regular mild solutions for a new class of fractional abstract integro-differential equations in Banach space.

    \item We present two examples in order to elucidate one of the results discussed in this present article.

    \item We emphasize that the preliminary results as well as the main results, it is not an easy and simple task, to obtain estimates involving the Mittag-Leffler functions in $\mathbb{B}_{1+\varepsilon}$. Consequently, controlling the $\varepsilon$-mild solution as well is not an easy task. These points enrich the work and make the results more solid and important.
\end{itemize}

Once the present results are finalized, some questions naturally arose, namely:
\begin{itemize}
    \item What are the next steps of this present work, i.e., if there are possible applications of the results obtained here, in particular, involving computational methods or just a numerical approach to the problem {\rm(\ref{Eq1})}?

    Although theoretical results were obtained, what are the next issues to be discussed?
\end{itemize}

A priori we do not have a direct physical application of the results investigated here, since it requires a little more care in the spaces worked and in the estimates involving the Mittag-Leffler functions when trying to apply the results in real problems. However, we are interested in investigating some issues of integro-differential equations using computational methods, as addressed in the following papers {\rm\cite{pal,pal1,pal2,pal4}}, which at the moment we believe to be more plausible. In particular, motivated by the work of Ma and Huang {\rm\cite{pal3}}, where they use various methods such as Adomian, differential transform method, collocation method and Taylor expansion approach to discuss fractional integro-differential equations and present some numerical examples. We believe that in the near future, these results can be discussed, they certainly contribute to soft computations. In addition, it is also worth highlighting another possible application, using an efficient local meshless collocation algorithm to approximate the fractional time evolution model that is applied for the modeling of heat flux in materials with memory {\rm\cite{jose}}. The model is based on the Riemann-Liouville fractional integral. Note that in this model, the Laplacian and the fractional integral appear, which behave very well with the problem studied here. In the next points to be addressed as future problems, we will highlight another memory equation. A priori for a possible computational approach, these are the objectives that can be traced for future work.

Furthermore, a natural continuation of the present work is expected, as several issues have arisen during the process. Also, motivated by the fact that problems are open in theory, we highlight some of them below:

\begin{enumerate}
\item  A natural continuation of the present paper, it is a possible extension of the {\rm{\bf Theorem} \ref{theo1}}.

\item Through {\rm{\bf Theorem} \ref{theo1}} we can investigate the continuation result and a blow-up alternative for the mild solutions {\rm Eq. (\ref{solutionmild})}.
    
\item We know that $S_{\alpha}(v)=\mathcal{E}_{\alpha}(-v^{\alpha} \mathcal{A})$, does not satisfy semigroup properties \cite{Peng10}, that is,
\begin{equation}\label{l21}
        S_{\alpha}(v)S_{\alpha}(s) \neq S_{\alpha}(v+s).
\end{equation}
    
Since $S_{\alpha}(v)$ does not satisfy semigroup properties {\rm Eq. (\ref{l21})}, and is part of the mild solutions of FDEs, in particular, of {\rm Eq. (\ref{Eq1})}. One of the open problems is to try to obtain a better property for {\rm Eq. (\ref{l21})}, and consequently, to discuss what consequences it brings to the results investigated here. Furthermore, certainly a better property $S_{\alpha}(v)$, will not only impact the results investigated here, but in general the entire literature involving FDEs.
    
    \item Another open question of great value and impact for the area is to investigate any existence, uniqueness, regularity or any other type of property of mild solutions for fractional differential and integro-differential equations in the sense of the $\psi$-HFD with sectorial operators. This question requires strong and high-caliber results of analysis and differential equations to solve it.
    
\end{enumerate}

Finally, we will present an approach on strongly damped plate equation with memory, an application of the results investigated in this present paper, which serve as a continuation of this article. We know that fractional operators, one of the fundamental properties, is the memory effect.

Consider the strongly damped plate equation with memory
\begin{eqnarray}\label{80}
    u_{tt}=-\Delta^{2} u+ \mu \Delta u_{y}+\int_{0}^{t} a(t-s)(-\Delta)^{\beta} u(s)ds+|u|^{p-1} u,\,t>0\,\,x\in\Omega
\end{eqnarray}
with the conditions $u(0,x)=u_{0}(x)$ and $u_{t}(0,x)=u_{1}(x)$, where $\mu>0$, $\Omega$ is an open sufficiently smooth subset of $\mathbb{R}^{n}$, $a:[0,\infty)\rightarrow[0,\infty)$ is defined by $a(t)=t^{v}$, $\Delta^{2}$ is bi-harmonic with hinged boundary conditions and $\Delta$ is the Laplacian with Dirichlet boundary conditions in $L^{2}(\Omega)$.

The idea is to rewrite the strongly damped plate equation with memory in the fractional version in such a way that it mirrors the problem discussed in the article. So in that sense, we can rewrite {\rm Eq. (\ref{80})}, as follows
\begin{equation*}
 \left\{ 
 \begin{array}{cll}
 ^{H}\der_{0+}^{\alpha,\beta}w(t) & = & {\bf A_{-1}}w+\displaystyle\int_{0}^{t} a(t-s) G(wt)ds+F(w),\,\,t>0\\ 
 I^{1-\gamma}_{0+}w(0) & = & w_0,
	 \end{array}
	 \right.
	 \end{equation*}
where $w=[u\,\, v]^{t}$, $\mathcal{A}_{-1}$ is the $\mathcal{F}^{-1}$-realization of $\mathcal{A}$, $G$ and $F$ are maps defined by
$G\begin{pmatrix}\psi_{1} \\
\psi_{2}  
\end{pmatrix}$
=
$\begin{pmatrix} D \\
(-\Delta)^{\beta} \psi_{2}
\end{pmatrix}$
and
$F\begin{pmatrix}\psi_{1} \\
\psi_{2}  
\end{pmatrix}$
=
$\begin{pmatrix} 0 \\
 |\psi_{2}|^{p-1} \psi_{2}
\end{pmatrix}$
with 
$\psi$=$\begin{pmatrix}\psi_{1} \\
\psi_{2} .
\end{pmatrix}$

Finally, we conclude that the results investigated here, in addition to being new, enabled further future research according to the points raised above. In this sense, we believe that the results have a positive impact on the area and open up new perspectives for the future.

\section*{Data Availability}

My manuscript has no associate data.

\section*{ Acknowledgements}

All authors’ contributions to this manuscript are the same. All authors read and approved the final manuscript. We are very grateful to the anonymous reviewers for their useful comments that led to improvement of the manuscript.


\end{document}